\documentclass{amsart}
\usepackage{amssymb}
\usepackage{amsbsy, amsthm, amsmath,amstext, amsopn}
\usepackage[all]{xy}
\usepackage{amsfonts}
\usepackage{amscd}
\usepackage{stmaryrd}
\newtheorem{thm}{Theorem} [section]

\newtheorem{corollary}[thm]{Corollary}
\newtheorem{prop}[thm]{Proposition}

\theoremstyle{definition}

\newtheorem{defn}[thm]{Definition}

\newtheorem{example}[thm]{Example}

\theoremstyle{remark}

\newtheorem{remark}[thm]{Remark}

\begin{document}

\numberwithin{equation}{section}

\newcommand{\hs}{\mbox{\hspace{.4em}}}
\newcommand{\ds}{\displaystyle}
\newcommand{\bd}{\begin{displaymath}}
\newcommand{\ed}{\end{displaymath}}
\newcommand{\bcd}{\begin{CD}}
\newcommand{\ecd}{\end{CD}}

\newcommand{\on}{\operatorname}
\newcommand{\proj}{\operatorname{Proj}}
\newcommand{\bproj}{\underline{\operatorname{Proj}}}

\newcommand{\spec}{\operatorname{Spec}}
\newcommand{\Spec}{\operatorname{Spec}}
\newcommand{\bspec}{\underline{\operatorname{Spec}}}
\newcommand{\pline}{{\mathbf P} ^1}
\newcommand{\aline}{{\mathbf A} ^1}
\newcommand{\pplane}{{\mathbf P}^2}
\newcommand{\aplane}{{\mathbf A}^2}
\newcommand{\coker}{{\operatorname{coker}}}
\newcommand{\ldb}{[[}
\newcommand{\rdb}{]]}

\newcommand{\Sym}{\operatorname{Sym}^{\bullet}}
\newcommand{\Symp}{\operatorname{Sym}}
\newcommand{\Pic}{\bf{Pic}}
\newcommand{\Aut}{\operatorname{Aut}}
\newcommand{\PAut}{\operatorname{PAut}}

\newcommand{\too}{\twoheadrightarrow}
\newcommand{\C}{{\mathbf C}}
\newcommand{\Z}{{\mathbf Z}}
\newcommand{\Q}{{\mathbf Q}}
\newcommand{\Cx}{{\mathbf C}^{\times}}
\newcommand{\Cbar}{\overline{\C}}
\newcommand{\Cxbar}{\overline{\Cx}}
\newcommand{\cA}{{\mathcal A}}
\newcommand{\cS}{{\mathcal S}}
\newcommand{\cZ}{{\mathcal Z}}
\newcommand{\cV}{{\mathcal V}}
\newcommand{\cM}{{\mathcal M}}
\newcommand{\bA}{{\mathbf A}}
\newcommand{\cB}{{\mathcal B}}
\newcommand{\cC}{{\mathcal C}}
\newcommand{\cD}{{\mathcal D}}
\newcommand{\D}{{\mathcal D}}
\newcommand{\cs}{{\mathbf C} ^*}
\newcommand{\boldc}{{\mathbf C}}
\newcommand{\cE}{{\mathcal E}}
\newcommand{\cF}{{\mathcal F}}
\newcommand{\bF}{{\mathbf F}}
\newcommand{\cG}{{\mathcal G}}
\newcommand{\G}{{\mathbb G}}
\newcommand{\cH}{{\mathcal H}}
\newcommand{\CI}{{\mathcal I}}
\newcommand{\cJ}{{\mathcal J}}
\newcommand{\cK}{{\mathcal K}}
\newcommand{\cL}{{\mathcal L}}
\newcommand{\baL}{{\overline{\mathcal L}}}

\newcommand{\Mf}{{\mathfrak M}}
\newcommand{\bM}{{\mathbf M}}
\newcommand{\bm}{{\mathbf m}}
\newcommand{\cN}{{\mathcal N}}
\newcommand{\theo}{\mathcal{O}}
\newcommand{\cP}{{\mathcal P}}
\newcommand{\cR}{{\mathcal R}}
\newcommand{\Pp}{{\mathbb P}}
\newcommand{\boldp}{{\mathbf P}}
\newcommand{\boldq}{{\mathbf Q}}
\newcommand{\bbL}{{\mathbf L}}
\newcommand{\cQ}{{\mathcal Q}}
\newcommand{\cO}{{\mathcal O}}
\newcommand{\Oo}{{\mathcal O}}
\newcommand{\cY}{{\mathcal Y}}
\newcommand{\OX}{{\Oo_X}}
\newcommand{\OY}{{\Oo_Y}}
\newcommand{\otY}{{\underset{\OY}{\ot}}}
\newcommand{\otX}{{\underset{\OX}{\ot}}}
\newcommand{\cU}{{\mathcal U}}\newcommand{\cX}{{\mathcal X}}
\newcommand{\cW}{{\mathcal W}}
\newcommand{\boldz}{{\mathbf Z}}
\newcommand{\qgr}{\operatorname{q-gr}}
\newcommand{\gr}{\operatorname{gr}}
\newcommand{\rk}{\operatorname{rk}}
\newcommand{\Sh}{\operatorname{Sh}}
\newcommand{\SH}{{\underline{\operatorname{Sh}}}}
\newcommand{\End}{\operatorname{End}}
\newcommand{\uEnd}{\underline{\operatorname{End}}}
\newcommand{\Hom}{\operatorname{Hom}}
\newcommand{\uHom}{\underline{\operatorname{Hom}}}
\newcommand{\uHomY}{\uHom_{\OY}}
\newcommand{\uHomX}{\uHom_{\OX}}
\newcommand{\Ext}{\operatorname{Ext}}
\newcommand{\bExt}{\operatorname{\bf{Ext}}}
\newcommand{\Tor}{\operatorname{Tor}}

\newcommand{\inv}{^{-1}}
\newcommand{\airtilde}{\widetilde{\hspace{.5em}}}
\newcommand{\airhat}{\widehat{\hspace{.5em}}}
\newcommand{\nt}{^{\circ}}
\newcommand{\del}{\partial}

\newcommand{\supp}{\operatorname{supp}}
\newcommand{\GK}{\operatorname{GK-dim}}
\newcommand{\hd}{\operatorname{hd}}
\newcommand{\id}{\operatorname{id}}
\newcommand{\res}{\operatorname{res}}
\newcommand{\lrar}{\leadsto}
\newcommand{\im}{\operatorname{Im}}
\newcommand{\HH}{\operatorname{H}}
\newcommand{\TF}{\operatorname{TF}}
\newcommand{\Bun}{\operatorname{Bun}}

\newcommand{\F}{\mathcal{F}}
\newcommand{\Ff}{\mathbb{F}}
\newcommand{\nthord}{^{(n)}}
\newcommand{\Gr}{{\mathfrak{Gr}}}

\newcommand{\Fr}{\operatorname{Fr}}
\newcommand{\GL}{\operatorname{GL}}
\newcommand{\gl}{\mathfrak{gl}}
\newcommand{\SL}{\operatorname{SL}}
\newcommand{\ff}{\footnote}
\newcommand{\ot}{\otimes}
\def\Ext{\operatorname {Ext}}
\def\Hom{\operatorname {Hom}}
\def\Ind{\operatorname {Ind}}
\def\bbZ{{\mathbb Z}}

\newcommand{\nc}{\newcommand}
\nc{\ol}{\overline} \nc{\cont}{\on{cont}} \nc{\rmod}{\on{mod}}
\nc{\Mtil}{\widetilde{M}} \nc{\wb}{\overline} \nc{\wt}{\widetilde}
\nc{\wh}{\widehat} \nc{\sm}{\setminus} \nc{\mc}{\mathcal}
\nc{\mbb}{\mathbb}  \nc{\K}{{\mc K}} \nc{\Kx}{{\mc K}^{\times}}
\nc{\Ox}{{\mc O}^{\times}} \nc{\unit}{{\bf \on{unit}}}
\nc{\boxt}{\boxtimes} \nc{\xarr}{\stackrel{\rightarrow}{x}}

\nc{\Ga}{\G_a}
 \nc{\PGL}{{\on{PGL}}}
 \nc{\PU}{{\on{PU}}}

\nc{\h}{{\mathfrak h}} \nc{\kk}{{\mathfrak k}}
 \nc{\Gm}{\G_m}
\nc{\Gabar}{\wb{\G}_a} \nc{\Gmbar}{\wb{\G}_m} \nc{\Gv}{G^\vee}
\nc{\Tv}{T^\vee} \nc{\Bv}{B^\vee} \nc{\g}{{\mathfrak g}}
\nc{\gv}{{\mathfrak g}^\vee} \nc{\BRGv}{\on{Rep}\Gv}
\nc{\BRTv}{\on{Rep}T^\vee}
 \nc{\Flv}{{\mathcal B}^\vee}
 \nc{\TFlv}{T^*\Flv}
 \nc{\Fl}{{\mathfrak Fl}}
\nc{\BRR}{{\mathcal R}} \nc{\Nv}{{\mathcal{N}}^\vee}
\nc{\St}{{\mathcal St}} \nc{\ST}{{\underline{\mathcal St}}}
\nc{\Hec}{{\bf{\mathcal H}}} \nc{\Hecblock}{{\bf{\mathcal
H_{\alpha,\beta}}}} \nc{\dualHec}{{\bf{\mathcal H^\vee}}}
\nc{\dualHecblock}{{\bf{\mathcal H^\vee_{\alpha,\beta}}}}
\newcommand{\ramBun}{{\bf{Bun}}}
\newcommand{\ramBuno}{\ramBun^{\circ}}

\nc{\Buntheta}{{\bf Bun}_{\theta}} \nc{\Bunthetao}{{\bf
Bun}_{\theta}^{\circ}} \nc{\BunGR}{{\bf Bun}_{G_\BR}}
\nc{\BunGRo}{{\bf Bun}_{G_\BR}^{\circ}}
\nc{\HC}{{\mathcal{HC}}}
\nc{\risom}{\stackrel{\sim}{\to}} \nc{\Hv}{{H^\vee}}
\nc{\bS}{{\mathbf S}}
\def\BRep{\operatorname {Rep}}
\def\Conn{\operatorname {Conn}}

\nc{\Vect}{{\operatorname{Vect}}}
\nc{\Hecke}{{\operatorname{Hecke}}}

\newcommand{\ZZ}{{Z_{\bullet}}}
\nc{\HZ}{{\mc H}\ZZ} \nc{\eps}{\epsilon}

\nc{\CN}{\mathcal N} \nc{\BA}{\mathbb A}

\nc{\ul}{\underline}

\nc{\bn}{\mathbf n} \nc{\Sets}{{\on{Sets}}} \nc{\Top}{{\on{Top}}}
\nc{\IntHom}{{\mathcal Hom}}

\nc{\Simp}{{\mathbf \Delta}} \nc{\Simpop}{{\mathbf\Delta^\circ}}

\nc{\Cyc}{{\mathbf \Lambda}} \nc{\Cycop}{{\mathbf\Lambda^\circ}}

\nc{\Mon}{{\mathbf \Lambda^{mon}}}
\nc{\Monop}{{(\mathbf\Lambda^{mon})\circ}}

\nc{\Aff}{{\on{Aff}}} \nc{\Sch}{{\on{Sch}}}

\nc{\bul}{\bullet}
\nc{\module}{{\operatorname{-mod}}}

\nc{\dstack}{{\mathcal D}}

\nc{\BL}{{\mathbb L}}

\nc{\BD}{{\mathbb D}}

\nc{\BR}{{\mathbb R}}

\nc{\BT}{{\mathbb T}}

\nc{\SCA}{{\mc{SCA}}}
\nc{\DGA}{{\mc DGA}}

\nc{\DSt}{{DSt}}

\nc{\lotimes}{{\otimes}^{\mathbf L}}

\nc{\bs}{\backslash}

\nc{\Lhat}{\widehat{\mc L}}

\nc{\QCoh}{QC}
\nc{\QC}{QC}
\nc{\Perf}{\rm{Perf}}
\nc{\Cat}{{\on{Cat}}}
\nc{\dgCat}{{\on{dgCat}}}
\nc{\bLa}{{\mathbf \Lambda}}

\nc{\BRHom}{\mathbf{R}\hspace{-0.15em}\on{Hom}}
\nc{\BREnd}{\mathbf{R}\hspace{-0.15em}\on{End}}
\nc{\colim}{\on{colim}}
\nc{\oo}{\infty}
\nc{\Mod}{{\on{Mod}}}

\nc\fh{\mathfrak h}
\nc\al{\alpha}
\nc\la{\alpha}
\nc\BGB{B\bs G/B}
\nc\QCb{QC^\flat}
\nc\qc{\on{QC}}

\nc{\fg}{\mathfrak g}

\nc{\Map}{\on{Map}} 

\nc{\fX}{\mathfrak X}

\nc{\XYX}{X\times_Y X}

\nc{\ch}{\check}
\nc{\fb}{\mathfrak b} \nc{\fu}{\mathfrak u} \nc{\st}{{st}}
\nc{\fU}{\mathfrak U}
\nc{\fZ}{\mathfrak Z}

\nc\fk{\mathfrak k} \nc\fp{\mathfrak p}

\nc{\BRP}{\mathbf{RP}} \nc{\rigid}{\text{rigid}}
\nc{\glob}{\text{glob}}

\nc{\cI}{\mathcal I}

\nc{\La}{\mathcal L}

\nc{\quot}{/\hspace{-.25em}/}

\nc\aff{\it{aff}}
\nc\BS{\mathbb S}

\nc\Loc{{\mc Loc}}
\nc\Tr{{\on{Tr}}}
\nc\Ch{{\mc Ch}}

\nc\ftr{{\mathfrak {tr}}}
\nc\fM{\mathfrak M}

\nc\Id{\operatorname{Id}}

\nc\bimod{\on{-bimod}}

\nc\ev{\operatorname{ev}}
\nc\coev{\operatorname{coev}}

\nc\pair{\operatorname{pair}}
\nc\kernel{\operatorname{kernel}}

\nc\Alg{\operatorname{Alg}}

\nc\init{\emptyset_{\text{\em init}}}
\nc\term{\emptyset_{\text{\em term}}}

\nc\Ev{\on{Ev}}
\nc\Coev{\on{Coev}}

\nc\es{\emptyset}
\nc\m{\text{\it min}}
\nc\M{\text{\it max}}
\nc\cross{\text{\it cr}}
\nc\tr{\on{tr}}
\nc\perf{\on{-perf}}
\nc\inthom{\mathcal Hom}
\nc\intend{\mathcal End}
\nc{\QSt}{{\mathcal QSt}}
\nc{\OO}{{\mathbb O}}

\nc\Ups{\Upsilon}

\title{Secondary traces}

\author{David Ben-Zvi}
\address{Department of Mathematics\\University of Texas\\Austin, TX 78712-0257}
\email{benzvi@math.utexas.edu}
\author{David Nadler}
\address{Department of Mathematics\\University
  of California\\Berkeley, CA 94720-3840}
\email{nadler@math.berkeley.edu}

\maketitle

\begin{abstract} 
We study an invariant, the secondary trace, attached to two commuting endomorphisms of a 2-dualizable object
in a symmetric monoidal higher category. We establish a secondary trace formula which encodes the natural symmetries of 
this invariant, identifying different realizations as an iterated trace. The proof consists of elementary Morse-theoretic arguments 
(with many accompanying pictures included) and may be seen as a concrete realization of the cobordism hypothesis with singularities on a marked 
2-torus. From this perspective, our main result identifies the secondary trace with two alternative presentations coming from the 
 the standard generators $S$ and $T$ of the mapping class group $SL_2(\Z)$.
 
We include two immediate consequences of the established invariance.  The first is a
modular invariance property for the 2-class function on a group arising as the 2-character of a categorical representation. 
The second is 
a generalization (for coherent sheaves or $\D$-modules) of the Atiyah-Bott-Lefschetz formula conjectured by Frenkel-Ng\^o 
in the case of a self-map of a smooth and proper  stack over a general base. 
\end{abstract}

 \tableofcontents


\section{Introduction}
We construct via elementary arguments a secondary trace formula encoding the natural symmetries of the secondary trace 
of two commuting right dualizable endomorphisms of a 2-dualizable object in a symmetric monoidal $(\oo, 2)$-category.
Independently of the formal assertions and proofs, we have found the many accompanying pictures extremely helpful in capturing the categorical constructions
and parallel Morse-theoretic analysis. We hope the interested reader could follow much of the paper like a picture book and only rely on the text when necessary.
Before turning to precise statements and immediate applications, let us comment on the inspiration and motivation for this work.

Since the influential preprint \cite{markarian} and the subsequent work
\cite{Cal1,Cal2,Ram,Ram2,Shk}, there have been many recent papers
\cite{Petit, lunts, polishchuk, cisinskitabuada} proving Riemann-Roch
and Lefschetz-type theorems in the context of differential graded categories and
Fourier-Mukai transforms. 
This paper places these results in 
the general formalism of traces in $\oo$-categories, and in
particular generalizes from dg categories to nonlinear homotopical settings.
There is a long history of abstract categorical advances~\cite{doldpuppe, may} including the recent~\cite{PS1, PS2},
culminating in our primary inspiration, Lurie's cobordism hypothesis with singularities~\cite{TFT}.
Lurie's theorem provides an exceptionally powerful unifying tool for higher
algebra (as well as a classification of extended topological field
theories with all possible defects). It provides a universal
refinement of graphical and pictorial calculi for
category theory, encoding  how higher
categories (with appropriate finiteness assumptions) are
representations of corresponding cobordism categories. The present paper
can be viewed as an elementary verification of a particular instance of the
cobordism hypothesis with singularities, using Morse theory to
factor an invariant categorical construction into elementary steps.

The primary motivation for this undertaking is the development of foundations
for ``homotopical harmonic analysis'' of group actions on categories,
aimed at decomposing derived categories of sheaves (rather than classical
function spaces) under the actions of natural operators. 
This undertaking
follows the groundbreaking path of Beilinson-Drinfeld within the geometric
Langlands program and is consonant with general themes in geometric
representation theory. 
The pursuit of a geometric analogue of the Arthur-Selberg trace formula
by Frenkel and Ng\^o~\cite{FN} has also been a source of inspiration and applications.
In this direction, we include in the final part of the introduction (Section~\ref{sect: app})
 a generalization of the Atiyah-Bott-Lefschetz formula for coherent sheaves or $\D$-modules for a self-map of a smooth and proper  stack over a general base.
Another inspiration is Lusztig's geometric theory of characters of finite group of Lie type, 
and in particular his nonabelian Fourier transform.
We expect secondary traces to have a variety of  applications in geometric
representation theory, which we plan to explore in forthcoming papers.

The companion paper \cite{primary} presents an alternative approach to Atiyah-Bott-Lefschetz formulas (without
the assumption of smoothness implicit here) via 
general functoriality properties of traces in higher categories.

\subsection{Symmetries of iterated traces}
Let $\cC$ be a symmetric monoidal $(\oo, 1)$-category with tensor $\otimes$ and unit $1_\cC$.
Given  a dualizable object $A\in\cC$, with dual denoted by $A^{op}$, the trace of an endomorphism $\Phi:A\to A$,
  is defined to be the compostion
$$\xymatrix{
\Tr_A(\Phi):
1_{\cA}\ar[r]^-{\Phi} & \End(A)\ar[r]^-{\sim} &  A\ot A^{op}
\ar[r]^-{\ev_A}& 1_\cA}
$$
One can visualize $\Tr_A(\Phi)$ by the picture
$$
\xymatrix{
 \ar@/_2pc/@{-}[d] A \ar[r]^-{\Phi} &A \ar@/^2pc/@{-}[d] \\
A^{op} \ar[r] &A^{op}\\
}
$$
where the arcs denote the coevaluation $\coev_A$ and evaluation $\ev_A$.
The special case $\Phi = \id_A$ provides the notion of the dimension of $A$. 

The key structure of the  trace is its cyclic symmetry.
Given dualizable objects $A, B\in \cC$ and a diagram of morphisms
$$\xymatrix{A  \ar@<+.5ex>[r]^\Phi &
  \ar@<+.5ex>[l]^{\Psi}B}$$ 
there is a cyclic symmetry
$$
\xymatrix{
m(\Phi, \Psi):\Tr_B(\Phi\circ \Psi) \ar[r]^-\sim & \Tr_A(\Psi \circ \Phi)
}
$$
One can visualize $m(\Phi, \Psi)$ by the sequence of equivalent pictures
$$
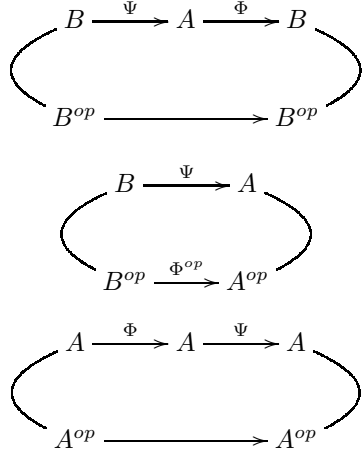
\xymatrix{
 \ar@/_2pc/@{-}[d] B \ar[r]^-{\Psi} & A \ar[r]^-{\Phi} &  B \ar@/^2pc/@{-}[d] \\
B^{op} \ar[rr]&& B^{op}\\
}
$$
$$
\xymatrix{
 \ar@/_2pc/@{-}[d] B \ar[r]^-{\Psi} &  A \ar@/^2pc/@{-}[d] \\
B^{op} \ar[r]^-{\Phi^{op}}& A^{op}\\
}
$$
$$
\xymatrix{
 \ar@/_2pc/@{-}[d] A \ar[r]^-{\Phi} & A \ar[r]^-{\Psi} &  A \ar@/^2pc/@{-}[d] \\
A^{op} \ar[rr]&& A^{op}\\
}
$$

Now let $\cA$ be a symmetric monoidal $(\oo, 2)$-category with tensor $\otimes$, unit $1_\cA$, and ``based loop" symmetric monoidal $(\oo, 1)$-category $\Omega\cA = \End_\cA(1_\cA)$. On the one hand,
we can forget the non-invertible $2$-morphisms and just as well discuss dualizable objects. On the other hand,
there is now the freedom to consider left and right dualizable 1-morphisms whose unit and counit 2-morphisms need not be
invertible.

Given dualizable objects $A, B \in \cA$, and
endomorphisms 
$\Phi:A\to A$, $\Phi': B\to B$, a right dualizable morphism $\Psi:A\to B$,
and a commuting transformation 
$$
\xymatrix{
\alpha:\Psi \circ \Phi \ar[r] & \Phi'\circ \Psi
}
$$
there is an induced trace map
$$
\xymatrix{
\varphi(\Psi, \alpha):\Tr_A(\Phi) \ar[r] &  \Tr_B(\Phi')
}
$$
built out of unit and counit morphisms, the commuting transformation and the cyclic symmetry.
One can visualize $\varphi(\Psi, \alpha)$ by the sequence of pictures
$$
\xymatrix{
 \ar@/_2pc/@{-}[d] A \ar[r]^-{\Phi} &  A \ar@/^2pc/@{-}[d] \\
A^{op} \ar[r] & A^{op}\\
}
$$
$$
\xymatrix{
 \ar@/_2pc/@{-}[d] A \ar[r]^-{\Phi} & A \ar[r]^-{\Psi} & B  \ar[r]^-{\Psi^r} &  A \ar@/^2pc/@{-}[d] \\
A^{op} \ar[rrr]&&& A^{op}\\
}
$$
$$
\xymatrix{
 \ar@/_2pc/@{-}[d] A \ar[r]^-{\Psi} & B \ar[r]^-{\Phi'} & B  \ar[r]^-{\Psi^r} &  A \ar@/^2pc/@{-}[d] \\
A^{op} \ar[rrr]&&& A^{op}\\
}
$$
$$
\xymatrix{
 \ar@/_2pc/@{-}[d] B \ar[r]^-{\Psi^r} & A \ar[r]^-{\Psi} & B  \ar[r]^-{\Phi'} &  B \ar@/^2pc/@{-}[d] \\
B^{op} \ar[rrr]&&& B^{op}\\
}
$$
$$
\xymatrix{
 \ar@/_2pc/@{-}[d] B \ar[r]^-{\Phi'} &  B \ar@/^2pc/@{-}[d] \\
B^{op} \ar[r] & B^{op}\\
}
$$

Now assume $A\in \cA$ is  a 2-dualizable (synonymously, fully dualizable) object in the sense that it is dualizable and is evaluation $\ev_A$ admits both left and right duals (synonymously, adjoints)~\cite{TFT}.
Then we have the following easy assertion.

\begin{prop}
Let $A$ be a 2-dualizable object of a symmetric monoidal $(\oo, 2)$-category $\cA$.

Let $\Phi:A\to A$ be a right dualizable endomorphism with right  adjoint $\Phi^r:A\to A$.

Then  $\Tr_A(\Phi) \in \Omega \cA$ is a dualizable object with dual $\Tr_A(\Phi)^\vee \simeq \Tr_A(\Phi^r) \in \Omega \cA$.
\end{prop}

The proposition opens up the possibility
 of forming a {\em secondary trace}. 
Suppose given right dualizable endomorphisms
$$
\xymatrix{
\Phi_a,\Phi_b: A\ar[r] & A
}
$$
and a commuting transformation 
$$
\xymatrix{
\alpha_{a, b}:\Phi_a\circ \Phi_b \ar[r] & \Phi_b\circ \Phi_a
}
$$

Then on the one hand, we have the  trace map
$$
\xymatrix{
\varphi(\Phi_a, \alpha_{a, b}):\Tr_A(\Phi_b) \ar[r] &  \Tr_A(\Phi_b)
}
$$
and  can form the iterated trace
$$
\Tr_{\Tr_A(\Phi_b)}(\varphi(\Phi_a, \alpha_{a, b}))   
\in \Omega\Omega\cA = \End_{\Omega \cA}(1_{\Omega\cA})
$$

On the other hand, we have the induced  transformation
$$
\xymatrix{
\alpha_{b, a^r} :\Phi_b\circ \Phi_a^r \ar[r]^-{\eta_{\Phi_a}} & 
\Phi_a^r\circ \Phi_a\circ \Phi_b\circ \Phi_a^r \ar[r]^-{\alpha_{a, b}} & 
\Phi_a^r\circ \Phi_b\circ \Phi_a\circ \Phi_a^r \ar[r]^-{\epsilon_{\Phi_a}} & 
\Phi_a^r\circ \Phi_b
}
$$
corresponding trace map
$$
\xymatrix{
\varphi(\Phi_b, \alpha_{b, a^r}):\Tr_A(\Phi_a^r) \ar[r] &  \Tr_A(\Phi_a^r)
}
$$
and  iterated trace
$$
\Tr_{\Tr_A(\Phi_a^r)}(\varphi(\Phi_b, \alpha_{b,a^r}))   
\in \Omega\Omega\cA = \End_{\Omega \cA}(1_{\Omega\cA})
$$

Here is the  main result of this paper. Its formulation is a generalization of \cite[Theorem 3.8]{polishchuk} from the setting of dg categories.

\begin{thm}[Secondary Trace Formula]\label{intro mainthm}
Let $A$ be a 2-dualizable object of a symmetric monoidal $(\oo, 2)$-category $\cA$.

Suppose given right dualizable endomorphisms
$$
\xymatrix{
\Phi_a,\Phi_b: A\ar[r] & A
}
$$
and a commuting transformation 
$$
\xymatrix{
\alpha_{a, b}:\Phi_a\circ \Phi_b \ar[r] & \Phi_b\circ \Phi_a
}
$$
there is a canonical equivalence
$$
\Tr_{\Tr_A(\Phi_b)}(\varphi(\Phi_a, \alpha_{a,b}))  \simeq 
\Tr_{\Tr_A(\Phi_a^r)}(\varphi(\Phi_b, \alpha_{b,a^r}))   
$$
\end{thm}

\begin{remark}[Shearing Formula]
We also can  deduce a  canonical equivalence
$$
\xymatrix{
 \varphi(\Phi_a \circ \Phi_b, \alpha_{a, b}')\simeq \varphi(\Phi_a, \alpha_{a. b}):\Tr(\Phi_b)\ar[r] & \Tr(\Phi_b)
}$$
and hence a canonical equivalence
$$
\xymatrix{
\Tr_{\Tr_A(\Phi_b)}( \varphi(\Phi_a \circ \Phi_b, \alpha_{a, b}'))
\simeq 
\Tr_{\Tr_A(\Phi_b)}
(\varphi(\Phi_a, \alpha_{a. b}))
}
$$
Here $\alpha_{a, b}'$ is the composition of $\alpha_{a, b}$ with the identity self-commuting transformation of $\Phi_a$.

\end{remark}

The geometric picture underlying the theorem and remark is that of a standard  2-torus with its meridians marked by $\Phi_a$ and $\Phi_b$
and their intersection point marked by $\alpha_{a, b}$. The different realizations of the secondary trace correspond to different ways to parse this picture into elementary pieces.  
We provide a detailed proof of Theorem \ref{intro mainthm}, applying Morse theory on the 2-torus to reduce the identifications to
a sequence of elementary moves. This is directly in the spirit of the cobordism hypothesis with singularities, from which one can deduce the theorem and remark. We nevertheless believe the
concrete arguments and accompanying pictures of this paper to be potentially illuminating to readers.

%
%

There are an infinite sequence of secondary trace invariants
labelled by the mapping class group $SL_2(\Z)$.
%
The three secondary trace invariants in the theorem and remark  correspond to the identity and the two standard generators $S,T \in SL_2(\Z)$. It follows from the cobordism hypothesis with singularities that all these invariants are canonically identified.
The techniques of this paper allow one to verify any one of these relations by hand (though one needs to be more careful with higher coherences).
In the next section, we mention a simple discrete setting where the identifications of the theorem and remark suffice to
establish full modular invariance.

\subsection{Fourier invariance of 2-characters} Fix a ground field $k$.

An immediate application of the secondary trace formula is to {\em 2-characters} of group actions on categories (as discussed in \cite{GK}).

As the source of 2-dualizable objects, we take smooth and proper dg categories (for example, derived categories
of coherent sheaves on smooth projective varieties). Or even more concretely, we take 2-vector spaces in the sense of finite module categories over 
$\Vect(k)$ (replacing varieties with finite sets), or semisimple $k$-linear categories with
finitely many simple objects.

Given a smooth and proper dg category $\cC$,  as the source of commuting pairs of endomorphisms, we look for groups $G$ acting on $\cC$ (for example, through an action on the underlying variety or set) and focus on
commuting pairs of elements of $G$. 

In this setting one can attach a function $\chi^{(2)}(\cC)$
on the set of conjugacy classes of pairs of commuting elements in $G$, which attaches to a pair $(g,h)$ 
the secondary trace 
$$
\chi^{(2)}(\cC)(g,h)=\Tr_{\Tr_\cC(h)}(g)
$$
corresponding to the action of $g$ on the categorical trace of $h$ acting on $\cC$.
Note that in this setting, the commuting transformation $\alpha$ is implicit in the $G$-action on $\cC$.

Observe that the 
 the groupoid of conjugacy classes of pairs of commuting elements $(g,h)$ is naturally identified with 
the groupoid $\Loc_{G}(T^2)$ of $G$-local systems on the 2-torus.
 It follows that $SL_2(\Z)$ 
naturally acts on the vector space of 2-class functions $k[\Loc_G(T^2)]$.

The secondary trace formula and shearing formula immediately imply the following symmetry for the generators $S, T \in SL_2(\Z)$.
Since the 2-class functions $k[\Loc_G(T^2)]$ form  a discrete vector space, the symmetry for all of $ SL_2(\Z)$ follows as well.

%


\begin{thm} The 2-character $\chi^{(2)}(\cC)$ is Fourier and shear invariant: for any commuting pair of elements
$(g,h)$, we have $$\chi^{(2)}(\cC)(g,h)=\chi^{(2)}(\cC)(h\inv,g)=\chi^{(2)}(\cC)(gh,h)$$ 
Thus $\chi^{(2)}(\cC)$ is
invariant under the natural $SL_2(\Z)$ action on 2-class functions $k[\Loc_G(T^2)]$.
\end{thm}

\begin{remark}
The terminology ``Fourier transform" for the interchange of $g$ and $h$ derives from Lusztig's Fourier transform
for finite groups, which is realized through the $SL_2(\Z)$-action on the fusion ring of $G$
(action of the mapping class group on the value of the corresponding Dijkgraaf-Witten 3d topological field theory on the 2-torus).
We plan to return to the relation to Lusztig's theory in future work. 
\end{remark}

\begin{remark}The above theorem  is easily verified when the $G$-category $\cC$ is of geometric origin.
For example, suppose a finite group $G$ acts on a finite set $X$ and we take $\cC=\Vect(X)$. Then to any commuting pair
$g,h\in G$ we can attach the orbifold $X^{g,h}$ of points of $X$ fixed by the pair $g,h$ (up to the action of 
the joint centralizer of the pair). Then we find that the 2-character $\chi^{(2)}(\cC)(g,h)$ simply measures the cardinality of $X^{g,h}$,
and hence is evidently modular invariant.
\end{remark}

\begin{remark}
The introduction of 2-class functions in \cite{GK} was inspired by the description by Hopkins-Kuhn-Ravenel \cite{HKR}
of $G$-characters in Morava K-theory $K_n$ of the classifying space $BG$ in terms of conjugacy classes of
commuting $n$-tuples of elements of $G$. The Fourier invariance result is analogous to the $\Aut(\Z_p^n)$-symmetry
of the HKR characters.
\end{remark}

\subsubsection{Topological field theory interpretation.} 
By the cobordism hypothesis, a 2-dualizable category $\cC$ defines a 2-dimensional framed topological field theory $\cZ_\cC$.
An action of a finite group $G$ on $\cC$ allows one to  gauge $\cZ_\cC$,  or in other words, to couple
$\cZ_\cC$ to $G$-local systems on framed 2-manifolds. 
Mathematically speaking, $\cC$ defines a 2-dualizable object in the 2-category of $G$-categories, or in other words, 
a boundary condition
in the 3-dimensional  Dijkgraaf-Witten theory $\cZ_G$ that assigns $G$-categories to a point.
Said another way,
$\cC$ defines a morphism from the vacuum 2-dimensional topological field theory to $\cZ_G$.
The 2-character $\chi^{(2)}(\cC)\in \cZ_G(T^2)=k[\Loc_G(T^2)]$  is the value of this morphism on the 2-torus. It is immediate that the morphism defined by $\cC$ is invariant under diffeomorphism,
so that $\chi^{(2)}(\cC)$ is invariant under the $SL_2(\Z)$-action on $\cZ_G(T^2)$.

The same reasoning applied to 1-dimensional topological field theory accounts for the cyclicity of the usual Chern character, i.e., the $S^1$-invariance of the trace map from $K$-theory to Hochschild homology (see \cite{TV} for a discussion).

Similar discussions apply to categories with algebraic or smooth  (synonomously, infinitesimally trivialized, strong, or Harish Chandra) actions of algebraic
groups $G$. Such categories
define boundary conditions in 3-dimensional gauge theories studied in \cite{BFN,character}, and their 
2-characters define modular invariant classes in the appropriate cohomological invariants of $\Loc_G(T^2)$ defined by the gauge theory.

\subsection{Applications in a traditional spirit}\label{sect: app}

Here we apply Theorem~\ref{intro mainthm} to deduce generalizations of traditional trace formulas.
We work in the context of derived algebraic geometry over a fixed base ring $R$.

\subsubsection{Atiyah-Bott Fixed Point Theorem}

Let $p:X\to \Spec R$ be a  smooth and proper stack so that $\Perf(X)$ is 2-dualizable as an $R$-linear small stable $\oo$-category.

%
%
%
%
Let $f:X\to X$ be an automorphism, and take
$$
\xymatrix{
\Phi_a:\Perf(X) \ar[r] & \Perf(X) &
\Phi_a(G) = f_*G
}
$$
It is invertible with inverse $\Phi_a^{-1}(G) = f^*G$.

Let $F \in \Perf(X)$, and take
$$
\xymatrix{
\Phi_b:\Perf(X) \ar[r] & \Perf(X) &
\Phi_b(G) = G \otimes_{\cO_X} F
}
$$
It is right dualizable with right adjoint
$$
\xymatrix{
\Phi^r_b:\Perf(X) \ar[r] & \Perf(X) &
\Phi^r_b(G) = G \otimes_{\cO_X} F^\vee
}
$$

Suppose $F$ is $f$-equivariant in the sense that we are given an equivalence
$$
\xymatrix{
\beta:f^*F\ar[r]^-\sim & F.
}$$
Observe that this is the same thing as
 a commuting transformation
$$
\xymatrix{
\alpha_{a,b}:\Phi_a\circ \Phi_b \ar[r]^-\sim & \Phi_b\circ \Phi_a.
}
$$

Now we calculate the first traces
$$
\xymatrix{
\Tr(\Phi_a^r) =   p_*\cO_{X^f}
&
\Tr(\Phi_b) = p_*\Delta^*\Delta_*F
}
$$
where $X^f$ denotes the derived fixed points of $f$ acting on $X$, and $\Delta:X\to X \times_{\Spec R} X$ denotes the diagonal.

Then we calculate that the trace morphism
$
\varphi(\Phi_a, \alpha_{a,b}) 
$
is simply the induced equivalence
$$
\xymatrix{
(f, \beta)_*:p_*\Delta^*\Delta_*F \ar[r]^-\sim &  p_*\Delta^*\Delta_*F
}
$$
And we calculate that the trace  morphism
$\varphi(\Phi_b, \alpha_{b,a^r} )$ 
is the composition
 $$
 \xymatrix{
 p_*\cO_{X^f} \ar[r]^-\eta &
  p_*((F\otimes F^\vee)|_{X^f})  
 \ar[r]^-{\beta \otimes \id} &
  p_*((F\otimes F^\vee)|_{X^f})  
 \ar[r]^-\eps &
   p_*\cO_{X^f}
  }
$$

Finally, Theorem~\ref{mainthm} provides the following equivalence
  which is a generalized form of the Atiyah-Bott Fixed Point Formula.

\begin{corollary} With the preceding setup, we have an identification
$$
 \Tr_{ p_*(F|_{X^f})}(\beta|_{X^f}) \simeq
  \Tr_{p_*\Delta^*\Delta_*F}((f, \beta)_*) \in R
  $$
\end{corollary}

In particular, if $p:X\to \Spec k $ is a smooth and proper scheme (so not a more general
 stack), and $f$ has isolated fixed points,
 then the fixed points $X^f$ have trivial derived structure, and the Hochschild-Kostant-Rosenberg Theorem
 provides a canonical identification 
 $$
 \Delta^*\Delta_*F \simeq F \otimes \Sym (\Omega_X[1]).
 $$ 
Setting $
M= F \otimes \Sym (\Omega_X[1])
$
with the $f$-equivariant structure $\beta \otimes f_*$ induced  by $\beta$ on $F$ and $f_*$ on $T_X$,
we obtain the traditional Atiyah-Bott Fixed Point Theorem
 $$
 \Tr_{ \Gamma(X^f, M\otimes \Sym (T_X[-1]) |_{X^f})}(\beta \otimes f_*|_{X^f})
 \simeq
 \Tr_{\Gamma(X, M)}((f, \beta)_*) \in k
 $$

\subsubsection{Lefschetz Fixed Point Theorem}
We will consider $\D$-modules though a similar story could be told for constructible sheaves.

All of the formal discussion of the previous section goes through with $\D$-modules in place of coherent sheaves,
assuming we fix constructibility assumptions so as to obtain a 2-dualizable category (the full category $\D(X)$ is smooth but 
not proper).

The result is the same formal identification
$$
 \Tr_{ p_*(F|_{X^f})}(\beta|_{X^f})
 \simeq
 \Tr_{p_*\Delta^*\Delta_*F}((f, \beta)_*)
 $$
 but now understood with $\D$-modules and $\D$-module functors. 
 
 For example, if $p:X\to \Spec k $ is a smooth and proper scheme (so not a more general
 stack), 
 then $X^f$ can be understood as the naive (underived) fixed points of $f$, and the composition $\Delta^*\Delta_*$ is the identity.
We obtain the traditional Lefschetz Fixed Point Theorem
 $$
 \Tr_{ \Gamma(X^f, F|_{X^f})}(\beta|_{X^f})
 \simeq
 \Tr_{\Gamma(X, F)}((f, \beta)_*) \in k
 $$

\subsection{Acknowledgements}

We are very grateful to Edward Frenkel for bringing to our attention
the question of an Atiyah-Bott Fixed Point Formula for stacks. He was
also instrumental in providing examples on which to test our
understanding. We are also grateful to Kevin Costello for a helpful
discussion on framings.


\section{Traces}

Our working setting is the higher category theory and algebra
developed by J.~Lurie~\cite{topos, HA, TFT, 2-cats}.

We will keep consistent notation and write  $\cC$ (resp. $\cA$) for a  symmetric
monoidal $(\oo, 1)$-category (resp. $(\oo, 2)$-category) with unit object $1_\cC$ (resp. $1_\cA$). 
We will
write $\Omega\cC= \End_\cC(1_\cC)$ (resp.   $\Omega\cA= \End_\cA(1_\cA)$) for the ``based loops" in $\cC$ (resp. $\cA$), or in other words,
the symmetric monoidal $\oo$-groupoid or equivalently space (resp. $(\oo, 1)$-category) of endomorphisms of
the monoidal unit $1_\cC$ (resp. $1_\cA$). 

\begin{example}[Algebras] Fix a symmetric monoidal $(\oo, 1)$-category $\cC$,
and let $\cA = \Alg(\cC)$ denote the Morita $(\oo,2)$-category of
algebras, bimodules, and intertwiners of bimodules within $\cC$.
The
forgetful map $\cA = \Alg(\cC) \to \cC$ is symmetric monoidal, and in
particular, the monoidal unit $1_\cA$ is the monoidal unit $1_\cC$
equipped with its natural algebra structure.
 Finally,
we have $\Omega\cA\simeq \cC$.  

For a specific example, one could take a commutative ring $k$ and $\cC =k\module$ the $(\oo,
1)$-category of complexes of $k$-modules. Then $\cA = \Alg(\cC)$ is the
$(\oo, 2)$-category of $k$-algebras, bimodules, and 
intertwiners of bimodules.
\end{example}

\begin{example}[Categories] A natural source of $(\oo,2)$-categories
is given by various theories of $(\oo,1)$-categories.  For example, for a commutative
ring $k$, one  could consider $\ul{\St}_k$, the $(\oo,2)$-category of
$k$-linear stable presentable $\oo$-categories, $k$-linear
continuous functors, and natural transformations.

Observe that $\Alg(k\module)$ is a full subcategory of $\ul{\St}_k$,
via the functor assigning to a $k$-algebra its stable presentable
$\oo$-category of modules. The essential image
consists of stable presentable categories admitting a  compact
generator.
\end{example}

\begin{remark}
Given a symmetric monoidal $(\oo, 1)$-category $\cC$, we can always
regard it as a symmetric monoidal $(\oo, 2)$-category $i(\cC)$ with
all $2$-morphisms invertible.

Conversely, given a symmetric monoidal $(\oo, 2)$-category $\cA$, we can always forget the non-invertible 2-morphisms
to obtain a symmetric monoidal $(\oo, 1)$-category $f(\cA)$. 

One can understand the above two operations as forming an adjoint pair $(i, f)$.
\end{remark}


\subsection{Dualizability}

\subsubsection{Dualizable objects}

\begin{defn}
(1)
An object $c$ of a symmetric monoidal $(\oo, 1)$-category $\cC$ is said to be {\em dualizable} if it admits a monoidal dual: there is a dual object $c^\vee \in \cC$
and evaluation and coevaluation morphisms
$$
\xymatrix{
\ev_c:c^\vee \otimes c\ar[r] & 1_\cC
&
\coev_c:1_\cC \ar[r] & c \otimes c^\vee
}$$
such that the usual compositions are  equivalent to the identity morphism
$$
\xymatrix{
c \ar[rr]^-{\coev_c \otimes \id_c} && c \otimes c^\vee \otimes c
\ar[rr]^-{ \id_c\otimes \ev_c} && c 
&
c^\vee \ar[rr]^-{ \id_{c^\vee} \otimes \coev_c} && c^\vee \otimes c \otimes c^\vee 
\ar[rr]^-{ \ev_c\otimes  \id_{c^\vee}} && c^\vee 
}
$$

(2)
An object $A$ of a symmetric monoidal $(\oo, 2)$-category $\cA$ is said to be {\em dualizable} if it is dualizable in the 
symmetric monoidal $(\oo, 1)$-category $f(\cA)$ obtained from $\cA$ by forgetting non-invertible 1-morphisms. In this case,
we will denote the dual object by $A^{op}\in \cA$.
\end{defn}

\begin{remark}[Duality and naiv\"et\'e in $\oo$-categories.]
It is a useful technical  observation that the notion of dualizability in the setting of $\oo$-categories is a ``naive" one: it
is a property of an object that can be checked in the underlying homotopy category. As a result, all of the 
categorical and 2-categorical calculations in this paper are similarly naive and explicit (and analogous to familiar unenriched categorical assertions), 
involving only small amounts of data that can
be checked by hand (rather than requiring higher coherences). 
We restrict ourselves only to assertions of this naive and accessible nature, specifying all maps that are needed rather than constructing
higher coherences (for which we view the cobordism hypothesis with singularities as the proper setting).
\end{remark}

\begin{example}
Any algebra object $A\in \Alg(\cC)$ is {dualizable} with dual the opposite algebra $A^{op}\in \Alg(\cC)$.
The evaluation morphism 
$$
\xymatrix{
\ev_A: A^{op} \otimes A \ar[r] & 1_\cC
}
$$
is given by $A$ itself regarded as an $A$-bimodule.
The coevaluation morphism 
$$
\xymatrix{
\coev_A:  1_\cC \ar[r] &A \otimes A^{op} 
}
$$
is also given by $A$ itself regarded as an $A$-bimodule.
\end{example}


\subsubsection{Dualizable morphisms}

Let us continue with $\cA$ a symmetric monoidal $(\oo, 2)$-category.
Consider two objects $A, B\in \cA$, and a morphism 
$$
\xymatrix{
\Phi:A \ar[r] & B.
}$$

\begin{example}
If $\cA = \Alg(\cC)$, then $\Phi$ is simply an $A^{op}\otimes B$-module.
\end{example}

If $B$ is dualizable with dual $B^{op}$, we can package $\Phi$ in the 
equivalent form
$$
\xymatrix{
\ev_\Phi = \ev_B\circ ( \id_{B^{op}}\otimes \Phi) :B^{op} \otimes A \ar[r] & 1_\cA
}
$$
If $A$ is dualizable with dual $A^{op}$, we can package $\Phi$ in the 
equivalent form
$$
\xymatrix{
\coev_\Phi = (\Phi\otimes \id_{A^{op}})\circ\coev_A:1_\cA \ar[r] & B \otimes A^{op}
}
$$

If both $A$ and $B$ are dualizble, 
we can also encode $\Phi$ by its dual morphism 
$$
\xymatrix{
\Phi^{op}:B^{op} \ar[rr]^-{\id_{B^{op}}\otimes \coev_A} && B^{op} \otimes A \otimes A^{op}
 \ar[rrr]^-{\id_{B^{op}}\otimes \Phi \otimes \id_{A^{op}}} &&&
  B^{op} \otimes B \otimes A^{op}
  \ar[rr]^-{\ev_B\otimes \id_{A^{op}}} && A^{op}
}
$$
which comes equipped with canonical equivalences
$$
\xymatrix{
(\Phi\otimes \id_{A^{op}})\circ\coev_A \simeq 
( \id_{B}\otimes \Phi^{op})\circ\coev_B : 1_\cA \ar[r] &  B \otimes A^{op}
}
$$
$$
\xymatrix{
\ev_B \circ ( \id_{B^{op}} \otimes \Phi)\simeq
\ev_A \circ (\Phi^{op} \otimes \id_A):  B^{op}\otimes A  \ar[r] &  1_\cA
}
$$

\begin{defn}
Let $A, B$ be objects of a symmetric monoidal $(\oo, 2)$-category $\cA$.

(1) A morphism $\Phi: A\to B$ is said to be {\em left dualizable}
 if it admits a left adjoint: there is a morphism $\Phi^\ell: B\to A$
and unit and counit morphisms
$$
\xymatrix{
\eta_ \Phi: \id_B \ar[r] & \Phi \circ \Phi ^\ell
&
\eps_\Phi : \Phi ^\ell \circ \Phi\ar[r] & \id_A
}$$
such that the usual compositions are  equivalent to the identity morphism
$$
\xymatrix{
\Phi^\ell \ar[rr]^-{ \id \otimes \eta_\Phi } && \Phi^\ell \circ\Phi \circ \Phi^\ell
\ar[rr]^-{\eps_ \Phi \otimes \id} && \Phi^\ell 
}$$
$$
\xymatrix{
\Phi \ar[rr]^-{\eta_\Phi  \otimes \id} &&  \Phi \circ \Phi^\ell \circ \Phi
\ar[rr]^-{\id \otimes \eps_ \Phi} && \Phi
}$$

(2) A morphism $\Phi: A\to B$ is said to be {\em right dualizable}
if it admits a right adjoint: there is a morphism $\Phi^r: B\to A$
and unit and counit morphisms
$$
\xymatrix{
\eta_ \Phi: \id_A \ar[r] & \Phi^r \circ \Phi 
&
\eps_\Phi : \Phi  \circ \Phi^r\ar[r] & \id_B
}$$
such that the usual compositions are  equivalent to the identity morphism
$$
\xymatrix{
\Phi^r \ar[rr]^-{ \eta_\Phi  \otimes \id} && \Phi^r \circ\Phi \circ \Phi^r
\ar[rr]^-{ \id\otimes\eps_ \Phi} && \Phi^r
}$$
$$
\xymatrix{
\Phi \ar[rr]^-{ \id \otimes\eta_\Phi } &&  \Phi\circ \Phi^r\circ \Phi
\ar[rr]^-{ \eps_ \Phi\otimes\id} && \Phi
}
$$
\end{defn}

%

\begin{remark}
Suppose $A, B\in \cA$ are dualizable. 

If $\Phi: A\to B$ is left  dualizable, then $\Phi^{op}: B^{op}\to A^{op}$ is right  dualizable with right adjoint 
$(\Phi^\ell)^{op}:A^{op}\to B^{op}$.
In other words, we have equivalent
adjoint pairs
$$
\xymatrix{
\Phi^\ell: B \ar@<0.7ex>[r] 
&
\ar@<0.7ex>[l]  A: \Phi
&
\Phi^{op}: B^{op} \ar@<0.7ex>[r] 
&
\ar@<0.7ex>[l]  A^{op}: (\Phi^\ell)^{op}
}
$$

Similarly,  if
$\Phi: A\to B$ is right  dualizable, then $\Phi^{op}: B^{op}\to A^{op}$ is left  dualizable with left adjoint 
$(\Phi^r)^{op}:A^{op}\to B^{op}$.
In other words, we have equivalent
adjoint pairs
$$
\xymatrix{
\Phi: A \ar@<0.7ex>[r] 
&
\ar@<0.7ex>[l]  B: \Phi^r
&
(\Phi^r)^{op}: A^{op} \ar@<0.7ex>[r] 
&
\ar@<0.7ex>[l]  B^{op}: \Phi^{op}
}
$$

\end{remark}

%


\subsubsection{2-dualizable objects}

\begin{defn}
An object $A$ of a symmetric monoidal $(\oo, 2)$-category $\cA$ is {\em 2-dualizable} (or fully dualizable) if it is dualizable
and its evaluation morphism
$$
\xymatrix{
\ev_A: A^{op} \otimes A \ar[r] & 1_\cA
}
$$
admits both a left adjoint and right adjoint. In this case, we will denote the left adjoint by $L_A$ and the right adjoint
by $R_A$.
\end{defn}

\begin{remark}
Here are the unit and counit compositions that
are equivalent to the identity
$$
\xymatrix{
L_A \ar[rr]^-{ \id \otimes \eta} && L_A \circ\ev_A \circ L_A
\ar[rr]^-{\eps \otimes \id} && L_A 
}$$
$$
\xymatrix{
\ev_A \ar[rr]^-{\eta \otimes \id} &&  \ev_A \circ L_A\circ \ev_A
\ar[rr]^-{\id \otimes \eps} && \ev_A 
}$$
$$
\xymatrix{
R_A \ar[rr]^-{ \eta \otimes \id} && R_A \circ\ev_A \circ R_A
\ar[rr]^-{ \id\otimes\eps} && R_A 
}$$
$$
\xymatrix{
\ev_A \ar[rr]^-{ \id \otimes\eta} &&  \ev_A \circ R_A\circ \ev_A
\ar[rr]^-{ \eps\otimes\id} && \ev_A 
}
$$
\end{remark}

\begin{remark}[Serre equivalence]
If  $A\in \cA$ is {2-dualizable}, we can package the left and right adjoints 
$$
\xymatrix{
L_A, R_A:1_\cA \ar[r] & A^{op} \otimes A
\simeq
A \otimes A^{op} 
}$$
in the equivalent form of endomorphisms $\ell_A, r_A: A\to  A$ so that
$$
\xymatrix{
L_A \simeq (\ell_A \otimes \id_{A^{op}})\circ \coev_A
&
R_A\simeq  (r_A \otimes \id_{A^{op}})\circ \coev_A
}
$$

It is straightforward to check that $\ell_A$ and $r_A$ are mutual inverse equivalences.
One often refers to $r_A$ as the {\em Serre equivalence} of $A$, and consequently $\ell_A$ as the {\em  inverse Serre equivalence}. One says that $A$ is {\em Calabi-Yau} if $r_A$, and hence $\ell_A$, is identified with the identity $\id_A$.
\end{remark}

\begin{remark}\label{coevad}
Alternatively, if $A\in \cA$ is 2-dualizable, we can consider the coevaluation morphism
$$
\xymatrix{
\coev_A:  1_\cA \ar[r] &A \otimes A^{op} 
}
$$
with its left and right adjoints
$$
\xymatrix{
L'_A,
R'_A: A\otimes A^{op} \simeq A^{op} \otimes A\ar[r] & 1_\cA
}
$$
given by the compositions
$$
\xymatrix{
L'_A = \ev_A\circ (r_A \otimes \id_{A^{op}})
&
R'_A = \ev_A\circ  (\ell_A \otimes \id_{A^{op}})
}
$$

Observe that there are canonical identifications
$$
\xymatrix{
R_A' \circ \coev_A\simeq \ev_A\circ L_A 
&
L_A' \circ \coev_A\simeq \ev_A\circ R_A 
}
$$
compatible with unit respectively counit morphisms, as well as 
 canonical identifications
$$
\xymatrix{
 \coev_A \circ R_A' \simeq (r_A\otimes\id_{A^{op}})\circ  L_A\circ \ev_A \circ  (\ell_A\otimes\id_{A^{op}})
}
$$
$$
\xymatrix{
L_A' \circ \coev_A \simeq (r_A \otimes\id_{A^{op}})\circ \ev_A\circ R_A \circ(\ell_A \otimes\id_{A^{op}})
}
$$
compatible with counit respectively unit morphisms.
\end{remark}

\begin{remark}\label{pairad}
Suppose $A\in \cA$ is 2-dualizable with respective left and right adjoints $L_A, R_A:1_\cA\to A^{op}\otimes A$ 
to the evaluation morphism $\ev_A:A^{op} \otimes A\to 1_\cA$.

Suppose $\Phi:A\to A$ is a left dualizable endomorphism.
Then
the pairing $\ev_\Phi: A^{op} \otimes A \to 1_\cA$ admits the left adjoint
$$
\xymatrix{
L_\Phi \simeq  (\Phi^\ell\otimes \id_{A^{op}})\circ L_A :1_\cA\ar[r] & A^{op} \otimes A
}
$$
since we simply have a composition of given adjoint pairs
$$
\xymatrix{
L_\Phi:1_\cA\ar@<0.7ex>[rr]^-{L_A} &&\ar@<0.7ex>[ll]^-{\ev_A}  A^{op} \otimes A \ar@<0.7ex>[rr]^-{\Phi^\ell\otimes\id_{A^{op}}} &&
\ar@<0.7ex>[ll]^-{\Phi\otimes\id_{A^{op}}} A^{op} \otimes A:\ev_\Phi
}
$$

Similarly, the pairing $\ev_{\Phi^\ell}: A^{op} \otimes A \to 1_\cA$ admits the right adjoint
$$
\xymatrix{
R_{\Phi^\ell} \simeq  (\Phi\otimes \id_{A^{op}}) \circ R_A :1_\cA\ar[r] & A^{op} \otimes A
}
$$
In other words, if $\Phi:A\to A$ is a right dualizable endomorphism, then
 the pairing $\ev_{\Phi}: A^{op} \otimes A \to 1_\cA$ admits the right adjoint
$$
\xymatrix{
R_{\Phi} \simeq  (\Phi^r\otimes \id_{A^{op}}) \circ R_A :1_\cA\ar[r] & A^{op} \otimes A
}
$$
\end{remark}


\subsection{Primary traces}


\begin{defn}\label{trace1}
Let $A$ be a dualizable object of a symmetric monoidal $(\oo, 1)$-category $\cC$.

The {\em trace} of an endomorphism $\Phi: A\to A$ is the element
$$
\Tr(\Phi)  = \Tr_A(\Phi) = \ev_A\circ (\Phi\otimes \id_{A^{\vee}}) \circ \coev_A \in \Omega\cC = \End_{\cC}(1_\cC)
$$

\end{defn}

\begin{remark}

We can visualize $\Tr(\Phi)$ by the picture
$$
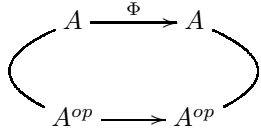
\xymatrix{
 \ar@/_2pc/@{-}[d] A \ar[r]^-{\Phi} &A \ar@/^2pc/@{-}[d] \\
A^{op} \ar[r] &A^{op}\\
}
$$
where the arcs denote the evaluation and coevaluation morphisms, and the unlabeled arrow denotes the identity. 

\end{remark}

\begin{example}
(1)
When $\Phi = \id_A$, the trace $\Tr(\id_A)$ is called the {\em Hochschild homology} of $A$.

(2)
When $A$ is a 2-dualizable object of a symmetric monoidal $(\oo, 2)$-category $\cA$,  the trace $\Tr(\ell_A)$ of the inverse Serre equivalence
is called the {\em Hochschild cohomology} of $A$.

(3)
As a special instance of Proposition~\ref{cylinders} below, the trace 
$\Tr(r_A)$ of the Serre equivalence is dualizable with dual $\Tr(r_A)^\vee \simeq \Tr(\ell_A)$.
\end{example}

\begin{example}
When $A = 1_\cC$ is the monoidal unit, and $\Phi:1_\cC \to 1_\cC$ is an endomorphism,
we have an evident equivalence $\Tr(\Phi) \simeq \Phi$ of endomorphisms of $1_\cC$.
\end{example}

\subsubsection{Cyclic symmetry}

The key structure of the trace is its cyclic symmetry whose most basic implication is the following.

\begin{defn}
Let $A, B$ be dualizable objects of a symmetric monoidal $(\oo, 1)$-category $\cC$.

Given a diagram of morphisms
$$\xymatrix{A  \ar@<+.5ex>[r]^\Phi &
  \ar@<+.5ex>[l]^{\Psi}B}$$ 
the {\em cyclic symmetry} is the equivalence
$$
\xymatrix{
m(\Phi, \Psi):\Tr(\Phi\circ \Psi) \ar[r]^-\sim & \Tr(\Psi \circ \Phi)
}
$$
given by the composition of duality equivalences
$$
\xymatrix{
 \ev_B\circ (( \Phi\circ \Psi) \otimes \id_B^{op}) \circ \coev_B\ar[r]^-\sim &
   \ev_A\circ (\Psi\otimes \Phi^{op}) \circ \coev_B\ar[r]^-\sim &
\ev_A\circ (( \Psi\circ \Phi) \otimes \id_A^{op}) \circ \coev_A
}
$$

\end{defn}

The cyclic symmetry is evidently functorial in $\Phi, \Psi$.

\begin{remark}

We can visualize $m(\Phi, \Psi)$ by the sequence of equivalent pictures
$$
\xymatrix{
 \ar@/_2pc/@{-}[d] B \ar[r]^-{\Psi} & A \ar[r]^-{\Phi} &  B \ar@/^2pc/@{-}[d] \\
B^{op} \ar[rr]&& B^{op}\\
}
$$
$$
\xymatrix{
 \ar@/_2pc/@{-}[d] B \ar[r]^-{\Psi} &  A \ar@/^2pc/@{-}[d] \\
B^{op} \ar[r]^-{\Phi^{op}}& A^{op}\\
}
$$
$$
\xymatrix{
 \ar@/_2pc/@{-}[d] A \ar[r]^-{\Phi} & A \ar[r]^-{\Psi} &  A \ar@/^2pc/@{-}[d] \\
A^{op} \ar[rr]&& A^{op}\\
}
$$

\end{remark}

\begin{remark}
We will be primarily interested in the  case $A=B$.
\end{remark}

\begin{example}
Set $A= B$.

Taking $\Phi= \id_A$ yields a canonical equivalence 
$$
\xymatrix{
\gamma_\Psi: \id_{\Tr(\Psi)} \ar[r]^-\sim & m( \id_A, \Psi) 
}
$$
and likewise, taking
$\Psi= \id_A$ yields a canonical equivalence 
$$
\xymatrix{
\gamma_\Phi: \id_{\Tr(\Phi)} \ar[r]^-\sim & m( \Phi, \id_A) 
}
$$

Thus taking $\Phi= \Psi =\id_A$ yields an automorphism of the identity of the Hochschild homology
$$
\xymatrix{
 (\gamma_\Psi)^{-1}\circ \gamma_\Phi:\id_{\Tr(\id_A)} \ar[r]^-\sim & \id_{\Tr(\id_A)}
}
$$
called the {\em BV homotopy}.
\end{example}


\subsubsection{Functoriality of trace}

\begin{defn}
Let $A, B$ be objects of an $(\oo, 1)$-category $\cC$. 

Given endomorphisms $\Phi:A\to A$,  $\Phi': B\to B$, and a morphism $\Psi:A\to B$,
 {\em a commuting transformation} is a morphism
$$
\xymatrix{
\alpha: \Psi \circ \Phi \ar[r] &\Phi'\circ \Psi 
}
$$
\end{defn}

\begin{remark}
One can view the pair $(\Psi, \alpha)$ as a morphism $(A, \Phi) \to (B, \Phi')$ in the (unbased)  loop category of $\cC$.
\end{remark}

\begin{defn}\label{def:trace funct}
Let $A, B$ be dualizable objects of a symmetric monoidal $(\oo, 1)$-category $\cC$. 

Given
endomorphisms 
$\Phi:A\to A$, $\Phi': B\to B$, a right dualizable morphism $\Psi:A\to B$,
and a commuting transformation 
$$
\xymatrix{
\alpha:\Psi \circ \Phi \ar[r] & \Phi'\circ \Psi
}
$$
the induced {\em trace map} is the morphism
$$
\xymatrix{
\varphi(\Psi, \alpha):\Tr(\Phi) \ar[r] &  \Tr(\Phi')
}
$$
given by the composition of the following morphisms.

First, we use the unit of duality
$$
\xymatrix{
\Tr(\Phi) \ar[r]^-{\eta_{\Psi}} &
\Tr(\Psi^r \circ \Psi\circ \Phi)
}
$$
Second, we use the commuting transformation
$$
\xymatrix{
\Tr(\Psi^r \circ \Psi \circ \Phi) 
\ar[r]^-{\alpha} &
\Tr(\Psi^r\circ \Phi' \circ \Psi) 
}
$$
Third we use the cyclic automorphism of the trace
$$
\xymatrix{
\Tr(\Psi^r\circ \Phi' \circ \Psi)
\ar[rr]^-{m(\Psi^r, \Phi'\circ  \Psi)} &&
\Tr(  \Phi' \circ\Psi\circ \Psi^r) 
}
$$
Finally, we 
 use the counit of duality 
$$
\xymatrix{
\Tr( \Phi' \circ\Psi\circ \Psi^r) 
 \ar[r]^-{\eps_{\Psi}} &
\Tr(\Phi')
}
$$
\end{defn}

\begin{remark}

We can visualize $\varphi(\Psi, \alpha)$ by the sequence of pictures
$$
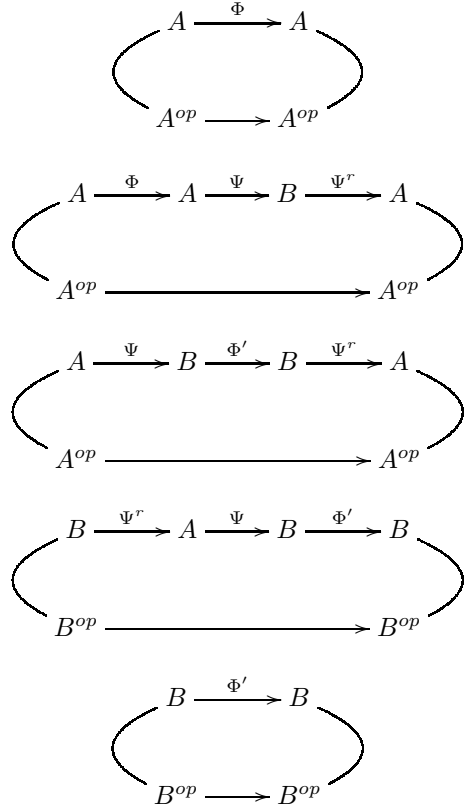
\xymatrix{
 \ar@/_2pc/@{-}[d] A \ar[r]^-{\Phi} &  A \ar@/^2pc/@{-}[d] \\
A^{op} \ar[r] & A^{op}\\
}
$$
$$
\xymatrix{
 \ar@/_2pc/@{-}[d] A \ar[r]^-{\Phi} & A \ar[r]^-{\Psi} & B  \ar[r]^-{\Psi^r} &  A \ar@/^2pc/@{-}[d] \\
A^{op} \ar[rrr]&&& A^{op}\\
}
$$
$$
\xymatrix{
 \ar@/_2pc/@{-}[d] A \ar[r]^-{\Psi} & B \ar[r]^-{\Phi'} & B  \ar[r]^-{\Psi^r} &  A \ar@/^2pc/@{-}[d] \\
A^{op} \ar[rrr]&&& A^{op}\\
}
$$
$$
\xymatrix{
 \ar@/_2pc/@{-}[d] B \ar[r]^-{\Psi^r} & A \ar[r]^-{\Psi} & B  \ar[r]^-{\Phi'} &  B \ar@/^2pc/@{-}[d] \\
B^{op} \ar[rrr]&&& B^{op}\\
}
$$
$$
\xymatrix{
 \ar@/_2pc/@{-}[d] B \ar[r]^-{\Phi'} &  B \ar@/^2pc/@{-}[d] \\
B^{op} \ar[r] & B^{op}\\
}
$$

\end{remark}

\begin{remark}
We will be primarily interested in the  case $A=B$ with $\Phi = \Phi'$.
\end{remark}

\begin{remark}\label{rem:alt trace funct}
Continuing with the setup of the  definition, we also have the induced morphism
$$
\xymatrix{
\beta: \Phi\circ \Psi^r\ar[r]^-{\eta_\Psi} & \Psi^r \circ \Psi\circ \Phi\circ \Psi^r\ar[r]^-\alpha 
& \Psi^r \circ \Phi'\circ \Psi\circ \Psi^r\ar[r]^-{\epsilon_\Psi} & \Psi^r \circ \Phi'
}
$$
The morphism $\varphi(\Psi, \alpha)$ is alternatively 
given by the composition of the following morphisms.

First, we use the unit of duality
$$
\xymatrix{
\Tr(\Phi) \ar[r]^-{\eta_{\Psi}} &
\Tr(\Phi\circ \Psi^r \circ \Psi)
}
$$
Second, we use the  morphism induced by the commuting transformation
$$
\xymatrix{
\Tr(\Phi\circ \Psi^r \circ \Psi) 
\ar[r]^-{\beta} &
\Tr(\Psi^r\circ \Phi' \circ \Psi) 
}
$$
Third we use the cyclic automorphism of the trace
$$
\xymatrix{
\Tr(\Psi^r\circ \Phi' \circ \Psi)
\ar[rr]^-{m(\Psi^r, \Phi'\circ  \Psi)} &&
\Tr(  \Phi' \circ\Psi\circ \Psi^r) 
}
$$
Finally, we 
 use the counit of duality 
$$
\xymatrix{
\Tr( \Phi' \circ\Psi\circ \Psi^r) 
 \ar[r]^-{\eps_{\Psi}} &
\Tr(\Phi')
}
$$

We can visualize this alternative presentation by the sequence of pictures
$$
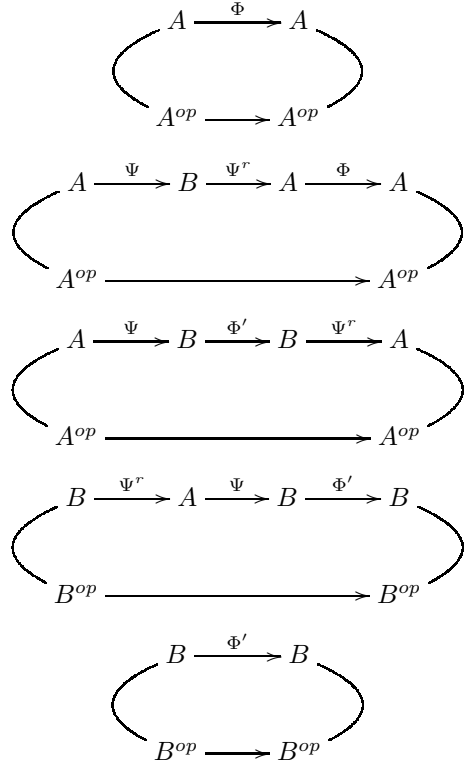
\xymatrix{
 \ar@/_2pc/@{-}[d] A \ar[r]^-{\Phi} &  A \ar@/^2pc/@{-}[d] \\
A^{op} \ar[r] & A^{op}\\
}
$$
$$
\xymatrix{
 \ar@/_2pc/@{-}[d] A \ar[r]^-{\Psi} & B \ar[r]^-{\Psi^r} & A  \ar[r]^-{\Phi} &  A \ar@/^2pc/@{-}[d] \\
A^{op} \ar[rrr]&&& A^{op}\\
}
$$
$$
\xymatrix{
 \ar@/_2pc/@{-}[d] A \ar[r]^-{\Psi} & B \ar[r]^-{\Phi'} & B  \ar[r]^-{\Psi^r} &  A \ar@/^2pc/@{-}[d] \\
A^{op} \ar[rrr]&&& A^{op}\\
}
$$
$$
\xymatrix{
 \ar@/_2pc/@{-}[d] B \ar[r]^-{\Psi^r} & A \ar[r]^-{\Psi} & B  \ar[r]^-{\Phi'} &  B \ar@/^2pc/@{-}[d] \\
B^{op} \ar[rrr]&&& B^{op}\\
}
$$
$$
\xymatrix{
 \ar@/_2pc/@{-}[d] B \ar[r]^-{\Phi'} &  B \ar@/^2pc/@{-}[d] \\
B^{op} \ar[r] & B^{op}\\
}
$$

To see that the above composition is naturally equivalent to 
$
\varphi(\Psi, \alpha)
$ 
is elementary and uses nothing more than basic  dualizability identities.
\end{remark}

%
%


\subsubsection{Dualizability of trace}

\begin{prop}\label{cylinders}
Let $A$ be a 2-dualizable object of a symmetric monoidal $(\oo, 2)$-category $\cA$.

Let $\Phi:A\to A$ be a right dualizable endomorphism with right  adjoint $\Phi^r:A\to A$.

Then  $\Tr(\Phi) \in \Omega \cA$ is a dualizable object with dual $\Tr(\Phi)^\vee \simeq \Tr(\Phi^r) \in \Omega \cA$.

\end{prop}

\begin{proof}
For future reference to the unit and counit morphisms of the duality, we give two equivalent dual proofs.

(1) Recall from Remark~\ref{coevad}
that the coevaluation morphism
$$
\xymatrix{
\coev_A:  1_\cA \ar[r] &A \otimes A^{op} 
}
$$
admits the right adjoint
$$
\xymatrix{
R'_A\simeq \ev_A\circ (\ell_A \otimes \id_{A^{op}}) : A^{op} \otimes A\ar[r] & 1_\cA
}
$$

Recall from Remark~\ref{pairad} that if $\Phi$ is right dualizable, then
the pairing $\ev_{\Phi}: A^{op} \otimes A \to 1_\cA$ admits the right adjoint
$$
\xymatrix{
R_{\Phi} \simeq  (\Phi^r\otimes \id_{A^{op}})\circ R_A :1_\cA\ar[r] & A^{op} \otimes A
}
$$

Composing adjoint pairs, we arrive at the adjoint pair 
$$
\xymatrix{
\ev_{\Phi}\circ \coev_A: 1_\cA \ar@<0.7ex>[r] &  \ar@<0.7ex>[l]  1_\cA:R_A' \circ R_{\Phi}
}
$$
Calculating the right hand side, we arrive at the dual pair
$$
\xymatrix{
\Tr(\Phi) = \ev_{\Phi}\circ \coev_A
& \Tr(\Phi^r)\simeq \Tr(\ell_A\circ \Phi^r \circ r_A) \simeq R_A' \circ R_{\Phi}
}$$

Finally, if $\Phi$ is left dualizable, apply the preceding argument starting with the right dualizable morphism $\Phi^\ell:A\to A$.

(2)
Recall from Remark~\ref{coevad}
that the coevaluation morphism
$$
\xymatrix{
\coev_A:  1_\cA \ar[r] &A \otimes A^{op} 
}
$$
admits the left adjoint
$$
\xymatrix{
L'_A\simeq \ev_A\circ (r_A \otimes \id_{A^{op}}) : A^{op} \otimes A\ar[r] & 1_\cA
}
$$

Recall from Remark~\ref{pairad} that 
the pairing $\ev_{\Phi^r}: A^{op} \otimes A \to 1_\cA$ admits the left adjoint
$$
\xymatrix{
L_{\Phi^r} \simeq  (\Phi\otimes \id_{A^{op}})\circ L_A :1_\cA\ar[r] & A^{op} \otimes A
}
$$

Composing adjoint pairs, we arrive at the adjoint pair 
$$
\xymatrix{
L'_A \circ L_{\Phi^r}: 1_\cA \ar@<0.7ex>[r] &  \ar@<0.7ex>[l]  1_\cA:\ev_{\Phi^r}\circ \coev_A
}
$$
Calculating the left hand side, we arrive at the dual pair
$$
\xymatrix{
\Tr(\Phi) \simeq \Tr(r_A\circ \Phi\circ \ell_A) \simeq L'_A \circ L_\Phi & \Tr(\Phi)=\ev_{\Phi^r}\circ \coev_A.
}$$
\end{proof}

\begin{remark}
The preceding proposition simply involves compositions of adjoint pairs.
The first proof  exhibits the coevaluation map of the duality as the composition 
of the units of the given adjunctions, 
while the second exhibits the evaluation map as
 the composition 
of the counits of the given adjunctions.
\end{remark}

\begin{remark}
For a left dualizable endomorphism $\Phi:A\to A$, we also have $\Tr(\Phi)^\vee\simeq \Tr(\Phi^\ell)$ by applying the proposition
to the right dualizable endmorphism $\Phi^\ell$.
\end{remark}

\begin{remark}\label{rem: trace funct + dual}
Let us combine  the functoriality and dualizability of traces. 

First, returning to the setup of Definition~\ref{def:trace funct}, let $A, B$ be dualizable objects of a symmetric monoidal $(\oo, 1)$-category $\cC$. 
Suppose given
endomorphisms 
$\Phi:A\to A$, $\Phi': B\to B$, a right dualizable morphism $\Psi:A\to B$,
and a commuting transformation 
$$
\xymatrix{
\alpha:\Psi \circ \Phi \ar[r] & \Phi'\circ \Psi
}
$$
so that we have
the trace map
$$
\xymatrix{
\varphi(\Psi, \alpha):\Tr(\Phi) \ar[r] &  \Tr(\Phi')
}
$$

Now assume in addition that $\Phi, \Phi'$ are themselves right dualizable. Then we have the induced commuting  transformation
$$
\xymatrix{
\alpha^r:\Psi^r \circ \Phi'^r  \ar[r] &  \Phi^r\circ \Psi^r
}
$$
and the corresponding trace map
$$
\xymatrix{
\varphi(\Psi^r, \alpha^r):\Tr(\Phi'^r) \ar[r] &  \Tr(\Phi^r)
}
$$

It is straightforward to check that under the identifications $\Tr(\Phi^r) \simeq \Tr(\Phi)^\vee$, $\Tr(\Phi'^r) \simeq \Tr(\Phi')^\vee$, there is a canonical equivalence of morphisms
$$ 
\xymatrix{
\varphi(\Psi^r, \alpha^r) \simeq \varphi(\Psi, \alpha)^\vee
}$$
\end{remark}



\subsection{Secondary traces}

Now we can state our main result. Its proof will occupy Section~\ref{factorsect} below.

Let $A$ be a 2-dualizable object of a symmetric monoidal $(\oo, 2)$-category $\cA$.

Suppose given right dualizable endomorphisms
$$
\xymatrix{
\Phi_a,\Phi_b: A\ar[r] & A
}
$$
and a commuting transformation 
$$
\xymatrix{
\alpha_{a, b}:\Phi_a\circ \Phi_b \ar[r] & \Phi_b\circ \Phi_a
}
$$

Then on the one hand, as in Definition~\ref{def:trace funct}, we have the  trace map
$$
\xymatrix{
\varphi(\Phi_a, \alpha_{a, b}):\Tr(\Phi_b) \ar[r] &  \Tr(\Phi_b)
}
$$
Since $\Tr(\Phi_b)$ is dualizable, we can form the iterated trace
$$
\Tr_{\Tr(\Phi_b)}(\varphi(\Phi_a, \alpha_{a, b}))   
\in \Omega\Omega\cA = \End_{\Omega \cA}(1_{\Omega\cA})
$$

On the other hand, as in Remark~\ref{rem:alt trace funct}, we can consider the induced  transformation
$$
\xymatrix{
\alpha_{b, a^r} :\Phi_b\circ \Phi_a^r \ar[r]^-{\eta_{\Phi_a}} & 
\Phi_a^r\circ \Phi_a\circ \Phi_b\circ \Phi_a^r \ar[r]^-{\alpha_{a, b}} & 
\Phi_a^r\circ \Phi_b\circ \Phi_a\circ \Phi_a^r \ar[r]^-{\epsilon_{\Phi_a}} & 
\Phi_a^r\circ \Phi_b
}
$$
the corresponding trace map
$$
\xymatrix{
\varphi(\Phi_b, \alpha_{b, a^r}):\Tr(\Phi_a^r) \ar[r] &  \Tr(\Phi_a^r)
}
$$
and
since $\Tr(\Phi_a^r)$ is dualizable, form the iterated trace
$$
\Tr_{\Tr(\Phi_a^r)}(\varphi(\Phi_b, \alpha_{b,a^r}))   
\in \Omega\Omega\cA = \End_{\Omega \cA}(1_{\Omega\cA})
$$

\begin{thm}\label{mainthm}
Let $A$ be a 2-dualizable object of a symmetric monoidal $(\oo, 2)$-category $\cA$.

Suppose given right dualizable endomorphisms
$$
\xymatrix{
\Phi_a,\Phi_b: A\ar[r] & A
}
$$
and a commuting transformation 
$$
\xymatrix{
\alpha_{a, b}:\Phi_a\circ \Phi_b \ar[r] & \Phi_b\circ \Phi_a
}
$$
there is a canonical equivalence
$$
\Tr_{\Tr(\Phi_b)}(\varphi(\Phi_a, \alpha_{a,b}))  \simeq 
\Tr_{\Tr(\Phi_a^r)}(\varphi(\Phi_b, \alpha_{b,a^r}))   
\in \Omega\Omega\cA = \End_{\Omega \cA}(1_{\Omega\cA})
$$
\end{thm}


\subsubsection{Shearing equivalence}

%
%

\begin{prop} 
Let $A$ be a 2-dualizable object of a symmetric monoidal $(\oo, 2)$-category $\cA$.

Suppose given a right dualizable endomorphism
$$
\xymatrix{
\Phi: A\ar[r] & A
}
$$
and the identity self-commuting structure 
$$
\xymatrix{
\id_{\Phi\circ \Phi}:\Phi\circ \Phi \ar[r]^-\sim & \Phi\circ \Phi
}
$$

Then the induced trace map is canonically equivalent to the identity
$$
\xymatrix{
\varphi(\Phi, \id_{\Phi\circ \Phi}) \simeq \id_{\Tr(\Phi)}:\Tr(\Phi)\ar[r] & \Tr(\Phi)
}$$

\end{prop}

\begin{proof}
Working with the identity commuting structure, we calculate the trace morphism
$$
\xymatrix{
\varphi(\Phi, \id_{\Phi\circ \Phi}): 
\Tr(\Phi)  \ar[r]^-{\eta_\Phi} & \Tr(\Phi^r\circ  \Phi\circ  \Phi) \ar[rr]^-{m(\Phi^r, \Phi\circ\Phi)}& & \Tr(\Phi\circ \Phi\circ \Phi^r ) \ar[r]^-{\epsilon_\Phi} & \Tr(\Phi)
}
$$

Using the cyclic symmetry of the trace, we can identify this composition with 
the functor $\Tr$ applied to the standard duality composition 
$$\xymatrix{
\id_\Phi: \Phi  \ar[r]^-{\eta_\Phi}  & \Phi\circ  \Phi^r\circ  \Phi  \ar[r]^-{\epsilon_\Phi} & \Phi
}
$$
\end{proof}

\begin{corollary}\label{switch} 
Let $A$ be a 2-dualizable object of a symmetric monoidal $(\oo, 2)$-category $\cA$.

Suppose given right dualizable endomorphisms
$$
\xymatrix{
\Phi_a,\Phi_b: A\ar[r] & A
}
$$
and a commuting transformation 
$$
\xymatrix{
\alpha_{a, b}:\Phi_a\circ \Phi_b \ar[r] & \Phi_b\circ \Phi_a
}
$$

Then there is
a canonical equivalence 
$$
\xymatrix{
 \varphi(\Phi_a \circ \Phi_b, \alpha_{a, b}')\simeq \varphi(\Phi_a, \alpha_{a. b}):\Tr(\Phi_b)\ar[r] & \Tr(\Phi_b)
}$$
where $\alpha_{a, b}'$ denotes the composition of $\alpha_{a, b}$ with the identity self-commuting structure of $\Phi_b$.
\end{corollary}

\begin{proof}
This is an easily checked instance of the functoriality of trace (see for example~\cite{primary} for general functoriality). By the above proposition, 
the composition $\Phi_a\circ \Phi_b$ with commuting structure  $\alpha_{a, b}'$ induces the trace morphism 
$\varphi(\Phi_a, \alpha_{a, b})\circ \id_{\Tr(\Phi_b)}$.
\end{proof}

%
%
%
%
%
%
%
%
%



\section{Factoring traces}\label{factorsect}

Let $A$ be a 2-dualizable object of a symmetric monoidal $(\oo, 2)$-category $\cA$.

Suppose given right dualizable endomorphisms
$$
\xymatrix{
\Phi_a,\Phi_b: A\ar[r] & A
}
$$
and a commuting transformation 
$$
\xymatrix{
\alpha_{a, b}:\Phi_a\circ \Phi_b \ar[r] & \Phi_b\circ \Phi_a
}
$$

Our aim here is to construct the canonical equivalence
$$
\Tr_{\Tr(\Phi_b)}(\varphi(\Phi_a, \alpha_{a,b}))  \simeq 
\Tr_{\Tr(\Phi_a^r)}(\varphi(\Phi_b, \alpha_{b,a^r}))   
\in \Omega\Omega\cA = \End_{\Omega \cA}(1_{\Omega\cA})
$$
asserted   in Theorem~\ref{mainthm}.

By definition, each side is the composition of independent elementary constructions.
We will explain how these constructions can be interwoven so that an equivalence of their compositions is evident.
As explained immediately below, it will be helpful to organize the elementary constructions
into a sequence of critical events.



\subsection{Organization of critical events}
We provide here an informal pictorial framework for organizing our future constructions, in the spirit of
topological field theory with defects (as formalized by the cobordism hypothesis with singularities).

Let $\BS =\{ z\in \C \, | \, |z| = 1\}$ denote the unit circle.
We will mark the point $i\in \BS$ and think of it as distinguished.
Consider the Morse function given by the real part of a complex number
$$
\xymatrix{
f:\BS\ar[r] & \BR
&
f(z) = \on{Re}(z)
}
$$
It has two usual critical points: a minimum at $-1\in \BS$ with critical value $f(-1) = -1$,
and a maximum at $1\in \BS$ with critical value $f(1) = 1$. We also think of 
the marked point $i\in \BS$ as a critical event with critical value $f(i) = 0$. 
We will picture marking the critical event with the transformations $\Phi_{k}$ for $k=a,b$,
and of the marked circle as representing the trace $\Tr(\Phi_k)$.

Now we will consider the product of two copies of the above setup. Let $\BS_1, \BS_2$ denote two copies of the circle $\BS$,
and let $\BT = \BS_1 \times \BS_2$
denote the two-dimensional torus. We consider it stratified by the marked submanifolds
$$
\xymatrix{
a = \{i\} \times \{ i \}
&
A_1 = \{i\} \times \BS_2
&
A_2 =  \BS_1 \times \{i\}
}
$$
We picture each of the circles as marked by one of the transformations $\Phi_k$, and
the intersection point $a$ as marked by the commuting transformation $\alpha_{ab}$.
The idea is now to assign a well defined invariant to the entire setup, which will be the secondary trace,
leading to the secondary trace formula when the invariant is parsed in different ways.

Consider the Morse function given by the sum of the Morse functions on each factor
$$
\xymatrix{
F:\BT\ar[r] & \BR
&
F(z_1, z_2) = f(z_1) + f(z_2) 
}
$$
It has four usual critical points: 
\begin{enumerate}
\item a minimum at $(-1, -1)$ with critical value $f(-1, -1) = -2$,
\item two saddles at $(-1, 1), (1, -1)$ with critical value $f(-1, 1) = f(1, -1)= 0$,
\item a maximum at $(1, 1)$ with critical value $f(1, 1) = 2$.
\end{enumerate}
We also think of the stratification as providing five further critical events:
\begin{enumerate}
\item the minimum $(i, -1)$ and maximum $(i, 1)$ of the restriction $F|_{A_1}$ with critical values $F(i, -1) = -1$
and $F(i, 1) = 1$ respectively,
\item the minimum $(-1, i)$ and maximum $(1, i)$ of the restriction $F|_{A_2}$ with critical values $F(-1, i) = -1$
and $F(1, i) = 1$ respectively,
\item the point $a = (i, i)$ with critical value $F(i, i) = 0$.
\end{enumerate}

In summary, there are 9 critical events given by the points $(z_1, z_2) \in \BT$ with $z_1, z_2\in \{-1, i, 1\}$
with 5 critical values $-2, -1, 0, 1, 2 \in \BR$. Above each of the critical values $-2, 2$, there is one critical event;
above each of the critical values $-1, 1$, there are two critical events; and
above the critical value $0$, there are three critical events. Note: though some critical events share the same critical value
in $\BR$,
all of the critical events are isolated in $\BT$. 

Here is a schematic picture summarizing the above discussion. 

$$
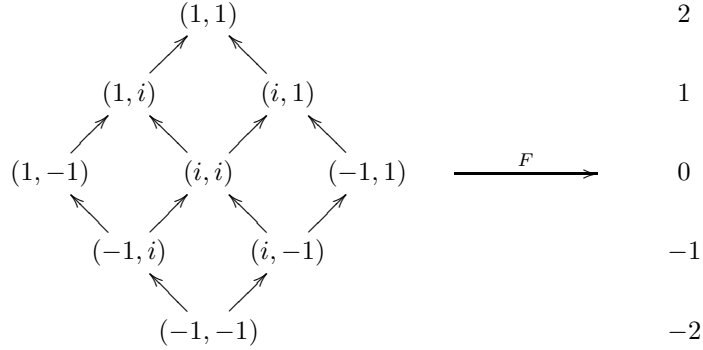
\xymatrix@!=0.5pc{
&& (1,1) && &&&& 2\\
& (1, i) \ar[ur]&&(i, 1)\ar[ul] && &&& 1 \\
(1,-1)\ar[ur] && (i,i) \ar[ul]\ar[ur]&& (-1,1) \ar[ul]& \ar[rr]^-F &&& 0\\
& (-1, i)\ar[ul] \ar[ur]&& (i, -1) \ar[ur]\ar[ul]& &&&& -1\\
&& \ar[ur]\ar[ul] (-1,-1) && &&&& -2\\
}
$$

We have added arrows between the critical points to organize them into a poset $\cP$. 
We will work with subsets $p \subset \cP$ that are saturated in the sense that whenever $x\in p$ and $y$ is less than $x$, then $y\in p$. 

Let $\cQ$ denote the category
whose objects are saturated subsets of $\cP$ and whose morphisms are given by inclusions.
It has an initial object given by the empty subset $\emptyset$ and a terminal object given by 
the entire poset $\cP$ itself.


\subsection{Critical transformations}

We continue with the category  $\cQ$ introduced immediately above and set $\cC = \Omega\cA = \End_{\cA}(1_\cA)$. 
 In this section, we explicitly construct a functor
$$
\xymatrix{
T:\cQ \ar[r] & \cC
}
$$
In particular, to the initial and terminal objects $\emptyset, \cP\in \cQ$,  we will assign the monoidal unit
$$T(\emptyset) = T(\cP) = 1_{\cC}.
$$ 
At the end of the day, evaluating $T$ on the unique morphism $\es\to \cP$, we will obtain an invariant 
$$
\xymatrix{
T(\emptyset, \cP)\in \Omega\Omega \cA = \End_{\cC}(1_{\cC}).
}
$$

We will then deduce Theorem~\ref{mainthm}
 by writing $\emptyset\to \cP$, and hence
the constructed endomorphism $T(\emptyset, \cP)$, 
as a composition in multiple ways. We will see that depending on our viewpoint,
we can equally regard $T(\emptyset, \cP)$ as either the left or right hand side of
Theorem~\ref{mainthm}.




\begin{example}[Secondary dimension]
To help guide the reader through the somewhat involved construction, we first quickly sketch the case where the endomorphisms
and the commuting transformation are all the identity. 
We will perform Morse theory on the  2-torus, with no markings, to illustrate
the secondary dimension $\dim (\dim(A))$. In order to keep track of framings and elementary moves, we draw this familiar picture in a less familiar way below.

First, at the index 0 critical point, we attach a disk whose framed boundary circle
represents the  Hochschild cohomology of $A$, realized as the trace of the inverse Serre equivalence $\ell_A$,
and apply the unit morphism $1_\cC\to \Tr(\ell_A)$.
In the next two parallel pictures, we alternatively attach a 1-handle in two different ways as we pass one of the two index 1 critical points. Both of these level sets (disjoint unions of two circles with the standard 2-framing) represent the object $\dim(A)\ot \dim(A)\simeq \dim(A)\ot \dim(A)^\vee$
and the composite morphism $1_\cC\to  \Tr(\ell_A) \to \dim(A)\ot \dim(A)^\vee$  is 
the coevaluation map of $\dim(A)$.
The next picture illustrates the framed level set obtained after passing through both index 1 critical  points. We may apply an isotopy to see that it represents the trace of the Serre equivalence $r_A$. Finally,
 at the index 2 critical point,
we cap off with a disk and find the resulting composite morphism
$ \dim(A)\ot \dim(A)^\vee \to \Tr(r_A) \to 1_\cC $  
  is 
the evaluation map of $\dim(A)$.


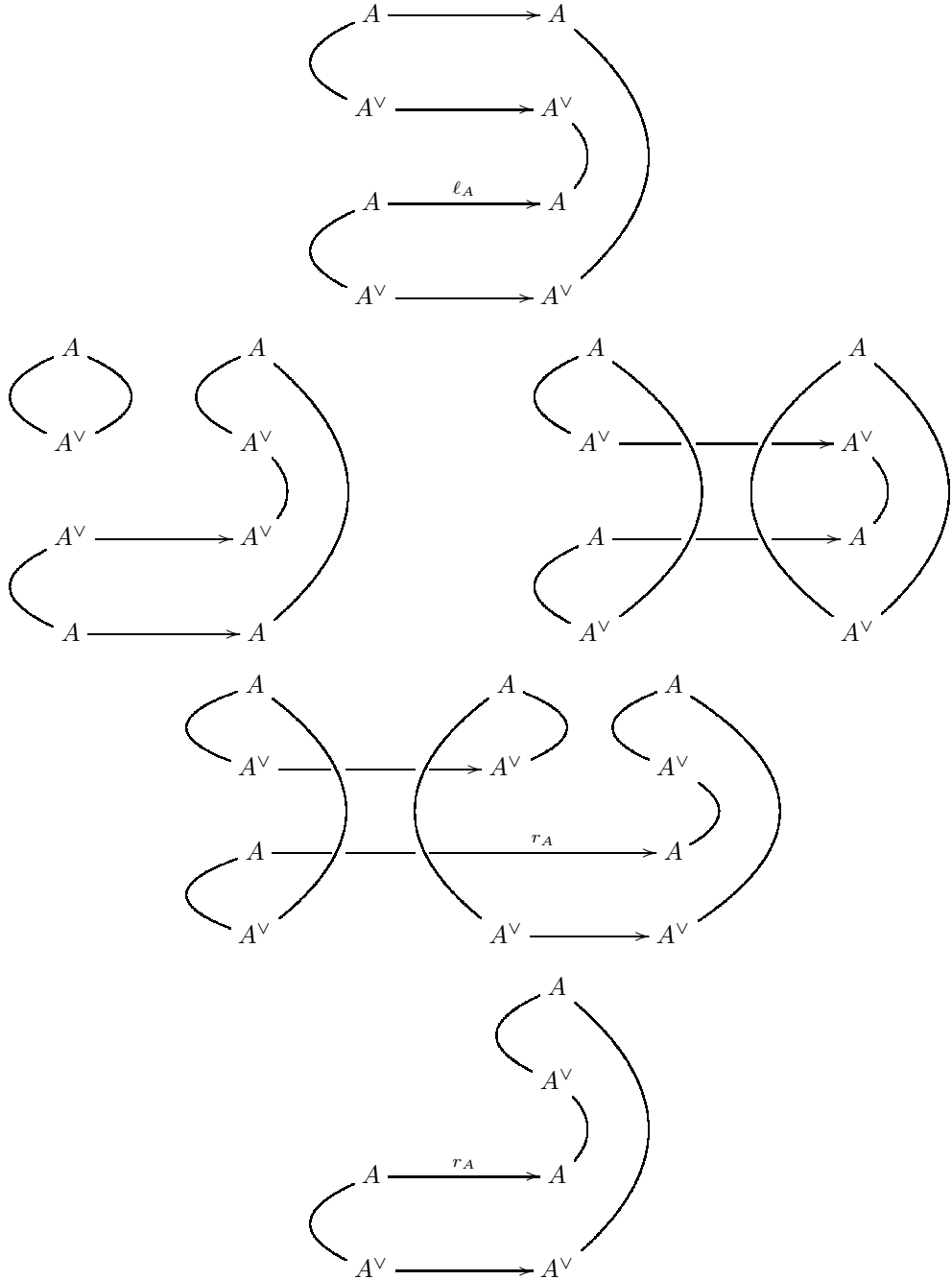
$$
\xymatrix{
\ar@/_2pc/@{-}[d] A \ar[rr]^-{} &&A \ar@/^3pc/@{-}[ddd] \\
A^{\vee} \ar[rr]^-{} &&A^{\vee} \ar@/^1pc/@{-}[d]\\
\ar@/_2pc/@{-}[d] A \ar[rr]^-{\ell_A} &&A\\
  A^{\vee}\ar[rr]^-{} &&A^{\vee}
}
$$

$$
\xymatrix{
\ar@/_2pc/@{-}[d] A  \ar@/^2pc/@{-}[d] && \ar@/_2pc/@{-}[d] 
A  \ar@/^3pc/@{-}[ddd]  &&&&
\ar@/_2pc/@{-}[d] A  \ar@/^3.5pc/@{-}[ddd] &&& \ar@/_3.5pc/@{-}[ddd]  
 A \ar@/^3pc/@{-}[ddd] \\
  A^{\vee} &  & A^{\vee} \ar@/^1pc/@{-}[d]&&&&
 A^{\vee} \ar'[r]'[rr][rrr] &&&A^{\vee}\ar@/^1pc/@{-}[d] \\
\ar@/_2pc/@{-}[d] A^{\vee} \ar[rr]&&A^{\vee}&&&&
\ar@/_2pc/@{-}[d] A \ar'[r]'[rr][rrr] &&& A\\
 A \ar[rr] && A  &&&&    A^{\vee}&  && A^{\vee}
}
$$

$$
\xymatrix@!=0.75pc{
\ar@/_2.25pc/@{-}[d]A  \ar@/^3pc/@{-}[ddd] 
&&& 
\ar@/_3pc/@{-}[ddd]  A   \ar@/^2pc/@{-}[d] && \ar@/_2pc/@{-}[d]   
   A \ar@/^3.5pc/@{-}[ddd]  \\ 
A^{\vee} \ar'[r]'[rr][rrr] &&& A^{\vee} && A^{\vee}   \ar@/^1.5pc/@{-}[d] \\
\ar@/_2.25pc/@{-}[d] A \ar'[r]'[rr][rrrrr]^-{r_A}&&&&& A\\
  A^{\vee}&  && A^{\vee}\ar[rr]& & A^{\vee} 
}
$$

$$
\xymatrix{
&&\ar@/_2pc/@{-}[d] A \ar@/^3pc/@{-}[ddd] \\
&&A^{\vee} \ar@/^1pc/@{-}[d]\\
\ar@/_2pc/@{-}[d] A \ar[rr]^-{r_A} &&A\\
  A^{\vee}\ar[rr]^-{} &&A^{\vee}
}
$$
\end{example}


\subsubsection{Global minimum}

Set $\m\in \cQ$ to be the saturated subset consisting of the global minimum $(-1,-1) \in \cP$ alone.

\begin{defn} We define the object 
$$T(\m) =  \ev_A \circ L_A 
$$ 
and the morphism
$$
\xymatrix{
T(\es, \m)= \eta:1_\cC \ar[r] &  \ev_A \circ L_A 
}
$$
 to be the unit of the adjunction.
\end{defn}

It will be useful to introduce the following reinterpretations of $T(\m)$.

\begin{remark}\label{firstreform}
We will work with the  canonical identifications
$$
\ev_A \circ L_A \simeq  \ev_A \circ (\id_A \otimes \ell^{op}_A) \circ \coev_A \simeq R'_A\circ \coev_A.
$$
which are compatible with the unit maps of the two adjoint pairs appearing.

With this understanding, we can picture $T(\m)$ in the form
$$
\xymatrix{
 \ar@/_2pc/@{-}[d] A \ar[r]^-{\id_{A}} &A \ar@/^2pc/@{-}[d] \\
A^{op} \ar[r]^-{\ell^{op}_A}&A^{op}\\
}
$$
where the arcs denote the evaluation and coevaluation morphisms. 
\end{remark}

\begin{remark}\label{reform}

Using the dualizability of $A$, 
we can write 
 $\ell_{A}^{op}$ as the composition
$$
\xymatrix{
\ell_{A}^{op}:A^{op} \ar[rr]^-{\id_{A^{op}} \otimes \coev_A} && A^{op} \otimes A \otimes A^{op}
\ar[rr]^{\id_A^{op} \otimes \ell_A \otimes\id_A^{op} } && A^{op} \otimes A \otimes A^{op}
\ar[rr]^-{\ev_A\otimes \id_{A^{op}}} &&
A^{op}
}
$$
Substituting this into $T(\m)$ as presented in Remark~\ref{firstreform}, allows us to rewrite it in the form
\begin{equation}\label{fourstrandabbrev}
T(\m) \simeq  \Ev  \circ \ell_3
\circ
\Coev
\end{equation}
where we have adopted the abbreviations
$$
\xymatrix{
\Ev =  \ev_{1, 4} \otimes \ev_{2, 3}
&
\Coev =  \coev_{1,2} \otimes \coev_{3, 4}
}
$$ 
Here the subscripts on the identity, inverse Serre equivalence, evaluations and coevaluations denote which strands they involve.
With this notation in mind, we will similarly adorn other functors 
with analogous subscripts to denote which strands they involve.
To simplify the notation further, 
we have also dropped identity morphisms from our formulas.

We can depict the above presentation of $T(\m)$ in the form
$$
\xymatrix{
\ar@/_2pc/@{-}[d] A \ar[r]^-{} &A \ar@/^3pc/@{-}[ddd] \\
A^{op} \ar[r]^-{} &A^{op} \ar@/^1pc/@{-}[d]\\
\ar@/_2pc/@{-}[d] A \ar[r]^-{\ell_A} &A\\
  A^{op}\ar[r]^-{} &A^{op}
}
$$

\end{remark}

\subsubsection{Stratified minima}

Set $\m_a, \m_b\in \cQ$ to be the saturated subsets 
$$
\m_a = \{(-1,-1), (i, -1)\}, \m_b=  \{(-1,-1), (i, -1)\} \subset \cP.
$$

Recall that for $k = a, b$, the endomorphism $\Phi_k:A \to A$ is right dualizable
with right adjoint $\Phi_k^r: A\to A$
and unit and counit morphisms
$$
\xymatrix{
\eta_ {\Phi_k}: \id_A \ar[r] & \Phi_k^r \circ \Phi_k
&
\eps_{\Phi_k} : \Phi_k \circ \Phi_k^r\ar[r] & \id_A
}$$

\begin{defn} For $k=a,b$, we define the object
$$
T(\m_k) =   \ev_A \circ ( (\Phi_k^r \circ \Phi_k ) \otimes \id_{A^{op}})\circ L_A
$$
\end{defn}

\begin{remark}
Let us place the above definition in the setting of Remark~\ref{reform},
and in particular in the notation of Equation~\ref{fourstrandabbrev}.

For $k=a,b$, we have a canonical equivalence
$$
T(\m_k) \simeq \Ev \circ ((\Phi_k^r \circ \Phi_k)_1\otimes \ell_{3})\circ\Coev
$$
which we can depict in the form
$$
\xymatrix{
\ar@/_2pc/@{-}[d] A \ar[r]^-{\Phi_k } &A \ar[r]^-{\Phi_k^r } & A\ar@/^3pc/@{-}[ddd]\\
A^{op} \ar[rr] &&A^{op}\ar@/^1pc/@{-}[d]\\
 \ar@/_2pc/@{-}[d]A \ar[rr]^-{\ell_A} &&A\\
  A^{op}\ar[rr]&&A^{op}
}
$$
\end{remark}

\begin{defn}
For $k=a,b$, we define the morphism
$$
\xymatrix{
T(\m, \m_k ): \ev_A \circ L_A \ar[r] &  
 \ev_A \circ ( (\Phi_k^r \circ \Phi_k ) \otimes \id_{A^{op}})\circ L_A
}
$$
 to be induced by the unit $\eta_{\Phi_k}$.
\end{defn}

\begin{remark}
With the above definition in hand, to continue to assemble the functor $T$, we obviously must take 
$$
\xymatrix{
T(\es, \m_k) = T(\m, \m_k) \circ T(\es, \m)
& k=a,b.
}
$$
We will often comment minimally on such evident aspects of the construction.
\end{remark}


\subsubsection{Crossing}
Set $\cross\in \cQ$ to be the saturated subset
$$
\cross = \{(-1,-1), (i, -1),  (i, -1), (i, i)\} \subset \cP.
$$

\begin{defn} We define the objects
$$
T(\m_a \cup \m_b) =   \ev_A \circ ( (\Phi_b^r \circ \Phi_b \circ\Phi_a^r  \circ \Phi_a ) \otimes \id_{A^{op}})\circ L_A
$$
$$
T(\cross) =   \ev_A \circ ( (\Phi_b^r \circ \Phi_a^r \circ\Phi_b  \circ \Phi_a ) \otimes \id_{A^{op}})\circ L_A
$$
\end{defn}

\begin{remark}
Let us place the above definition in the setting of Remark~\ref{reform}.

We have canonical equivalences
$$
T(\m_a \cup \m_b) \simeq \Ev \circ ((\Phi_b^r \circ \Phi_b \circ\Phi_a^r  \circ \Phi_a )_1 \otimes \ell_{3})\circ\Coev
$$
$$
T(\cross) \simeq \Ev \circ ((\Phi_b^r \circ \Phi_a^r \circ\Phi_b  \circ \Phi_a )_1 \otimes \ell_{3})\circ\Coev
$$
which we can depict in the following form. First, for $T(\m_a \cup \m_b)$, we have the picture
$$
\xymatrix{
\ar@/_2pc/@{-}[d] A \ar[r]^-{\Phi_a } &A \ar[r]^-{\Phi_a^r } & A \ar[r]^-{\Phi_b } &
A \ar[r]^-{\Phi_b^r }  &A\ar@/^3pc/@{-}[ddd]\\
A^{op} \ar[rrrr] &&&&A^{op}\ar@/^1pc/@{-}[d]\\
  \ar@/_2pc/@{-}[d] A \ar[rrrr]^-{\ell_A}&&&&A\\
  A^{op}\ar[rrrr] &&&&A^{op}
}
$$
Second, for $T(\cross)$, we have the picture
$$
\xymatrix{
\ar@/_2pc/@{-}[d]  A \ar[r]^-{\Phi_a } &A \ar[r]^-{\Phi_b } & A \ar[r]^-{\Phi_a^r } &A \ar[r]^-{\Phi_b^r } 
&A \ar@/^3pc/@{-}[ddd]\\
A^{op} \ar[rrrr] &&&&A^{op}\ar@/^1pc/@{-}[d]\\
\ar@/_2pc/@{-}[d]  A \ar[rrrr]^-{\ell_A}&&&&A\\
  A^{op}\ar[rrrr] &&&&A^{op}
}
$$
\end{remark}

Now recall  the given commuting transformation 
$$
\xymatrix{
\alpha_{a,b}:\Phi_a\circ \Phi_b \ar[r] & \Phi_b\circ \Phi_a
}
$$
In Remark~\ref{rem:alt trace funct}, we explained how $\alpha_{a, b}$ induces a commuting transformation
$$
\xymatrix{
\alpha_{b, a^r}:\Phi_b \circ \Phi_a^r
\ar[r] &    \Phi_a^r \circ \Phi_b 
}
$$

\begin{defn}
For $k = a,b$,
we define the morphism
$$
\xymatrix{
T(\m_k, \m_a\cup \m_b): T(\m_k) \ar[r] & T(\m_a\cup \m_b)
}
$$
 to be induced by the unit $\eta_{\Phi_{k'}}$ where $k'\in \{a,b\}$ is not equal to $k$.

We define the morphism
$$
\xymatrix{
T(\m_a\cup \m_b, \cross):  T(\m_a\cup \m_b) \ar[r] & T(\cross)
}
$$
to be induced by $\alpha_{b, a^r}$.
\end{defn}

\begin{remark}
To continue to assemble the functor $T$, we take all compositions to be what they need to be by our previous definitions.
With our definitions so far, there is no ambiguity due to the obvious equivalence
$$
T(\m_a, \m_a\cup \m_b) \circ T(\m, \m_a) \simeq T(\m_b, \m_a\cup \m_b) \circ  T(\m, \m_b)
$$
\end{remark}



\subsubsection{Saddles}

Set $s_a, s_b\in \cQ$ to be the saturated subsets
$$
s_a = \{(-1,-1), (i, -1),  (1, -1)\}, s_b = \{(-1,-1), (-1, i),  (-1, 1)\} \subset \cP.
$$

\begin{defn} We define the objects
$$
T(s_a) =   \Ev\circ  ( \Phi_a^r)_1 \circ (\coev_{1,4} \circ L'_{1,4}) \circ \ell_3 \circ 
(\Phi_a)_1\circ \Coev
$$
$$
T(s_b) =   \Ev\circ  (\Phi_b^r)_1\circ (R_{1,2} \circ \ev_{1,2}) \circ \ell_3 \circ (\Phi_b)_1\circ \Coev
$$
\end{defn}

\begin{defn} We define the morphisms
$$
\xymatrix{
T(\m_a, s_a ): T(\m_a) \ar[r] & T(s_a)  
}
$$
$$
\xymatrix{
T(\m_a, s_a ): T(\m_a) \ar[r] & T(s_b)  
}
$$
to be induced by the respective units of adjunctions
$$
\xymatrix{
\eta_{1,4}:\id_{1,4} \ar[r] & \coev_{1,4} \circ L'_{1,4}
&
\eta_{1,2}:\id_{1,2} \ar[r] & R_{1,2} \circ \ev_{1,2}
}$$
\end{defn}

\begin{remark}

Using nothing but  the basic properties of $\ev_A, \coev_A$, and canceling
the Serre equivalence with its inverse,
we have  the following evident simplifications
$$
T(s_a) \simeq  \Ev\circ  ( \Phi_a^r)_1\circ (\coev_{1,4} \circ \ev_{1,4})\circ (\Phi_a)_1\circ \Coev
$$
$$
T(s_b) \simeq   \Ev\circ  ( \Phi_b^r)_1\circ (\coev_{1,2} \circ \ev_{1,2}) \circ (\Phi_b)_1\circ \Coev
$$

%

Let us draw four-strand pictures  of the above objects analogous to our previous pictures.
As before, we will represent evaluations and coevaluations by arcs,
and leave strands representing identity maps unlabeled. 
First, we can depict $T(s_a)$ by the picture
$$
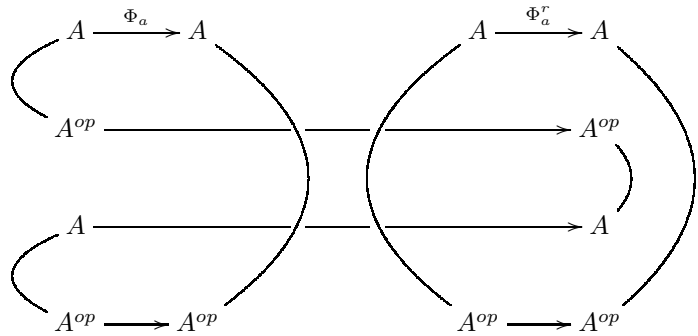
\xymatrix{
\ar@/_2pc/@{-}[d] A \ar[r]^-{\Phi_a } &A \ar@/^3.5pc/@{-}[ddd] &&& \ar@/_3.5pc/@{-}[ddd]  
 A \ar[r]^-{\Phi^r_a }  & A\ar@/^3pc/@{-}[ddd] \\
A^{op} \ar'[rr]'[rrr][rrrrr] &&&&&A^{op}\ar@/^1pc/@{-}[d]\\
\ar@/_2pc/@{-}[d] A \ar'[rr]'[rrr][rrrrr] &&&&& A\\
  A^{op}\ar[r] & A^{op} &  && A^{op}\ar[r]& A^{op} 
}
$$
Second, we can depict $T(s_b)$ by the picture
$$
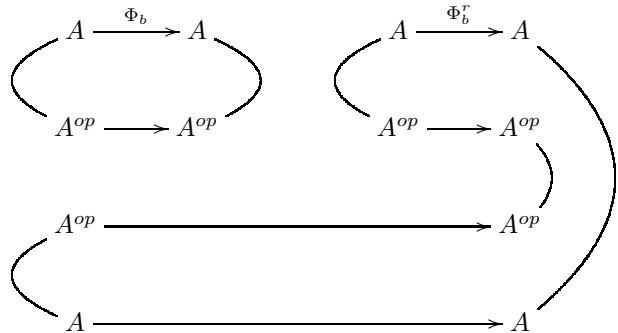
\xymatrix{
\ar@/_2pc/@{-}[d] A \ar[r]^-{\Phi_b } &A \ar@/^2pc/@{-}[d] && \ar@/_2pc/@{-}[d]  
A \ar[r]^-{\Phi_b^r }  &
A  \ar@/^3pc/@{-}[ddd]  \\
  A^{op}\ar[r] & A^{op} &  & A^{op}\ar[r]& A^{op} \ar@/^1pc/@{-}[d]\\
\ar@/_2pc/@{-}[d] A^{op} \ar[rrrr]&&&&A^{op}\\
 A \ar[rrrr] &&&& A
}
$$
\end{remark}

Now we can further define the functor $T$ on the other objects and morphisms of $\cQ$ that are unions
of those we have already encountered. 
In the following definition, we first state a formal expression for the value of the functor $T$ on such  an object, and then depict 
the formal expression with a 
four-strand picture.

\begin{defn}\label{defofsomeunions}

$$
T(s_a\cup\m_b)=  \Ev\circ   (\Phi_b^r \circ \Phi_b\circ \Phi_a^r)_1\circ (\coev_{1,4} \circ \ev_{1,4}) \circ (\Phi_a)_1\circ \Coev
$$
$$
\xymatrix{
\ar@/_2.25pc/@{-}[d]A \ar[r]^-{\Phi_a } &A \ar@/^3.5pc/@{-}[ddd] 
&&& 
\ar@/_3.5pc/@{-}[ddd]   A \ar[r]^-{\Phi_a^r }  &A \ar[r]^-{\Phi_b }  &A \ar[r]^-{\Phi_b^r }  &A 
  \ar@/^3.5pc/@{-}[ddd]\\
A^{op} \ar'[rr]'[rrr][rrrrrrr] &&&&&&&A^{op} \ar@/^1.5pc/@{-}[d]\\
\ar@/_2.25pc/@{-}[d] A \ar'[rr]'[rrr]  [rrrrrrr]&&&&&&& A \\
  A^{op}\ar[r] & A^{op} &  && A^{op}\ar[rrr]& &&A^{op} 
}
$$

$$
T(\m_a \cup s_b) =   \Ev\circ  ( \Phi_b^r)_1\circ (\coev_{1,2} \circ \ev_{1,2})  \circ 
(\Phi_b \circ\Phi_a^r \circ\Phi_a)_1\circ \Coev
$$
$$
\xymatrix{
\ar@/_2pc/@{-}[d]A \ar[r]^-{\Phi_a } & A \ar[r]^-{\Phi^r_a} & A \ar[r]^-{\Phi_b } &
A \ar@/^2pc/@{-}[d] && \ar@/_2pc/@{-}[d]   A \ar[r]^-{\Phi_b^r }   &A  \ar@/^3.5pc/@{-}[ddd]\\
  A^{op}\ar[rrr] && & A^{op}  && A^{op}\ar[r]& A^{op}  \ar@/^1.5pc/@{-}[d]\\
\ar@/_2pc/@{-}[d] A^{op} \ar[rrrrrr]&&&&&&A^{op}\\
 A \ar[rrrrrr] &&&&&& A
}
$$

$$
T(s_a\cup \cross)=  \Ev\circ  ( \Phi_b^r \circ \Phi_a^r\circ \Phi_b)_1\circ (\coev_{1,4} \circ \ev_{1,4}) \circ (\Phi_a)_1
\circ \Coev
$$
$$
\xymatrix{
\ar@/_2pc/@{-}[d]A \ar[r]^-{\Phi_a } &A \ar@/^3.5pc/@{-}[ddd] 
&&& 
\ar@/_3.5pc/@{-}[ddd]  A \ar[r]^-{\Phi_b }  &A \ar[r]^-{\Phi_a^r }  &A \ar[r]^-{\Phi_b^r }  &A 
 \ar@/^3.5pc/@{-}[ddd] \\
A^{op} \ar'[rr]'[rrr][rrrrrrr] &&&&&&&A^{op}  \ar@/^1.5pc/@{-}[d]\\
 \ar@/_2pc/@{-}[d] A \ar'[rr]'[rrr][rrrrrrr] &&&&&&& A\\
  A^{op}\ar[r] & A^{op} &  && A^{op}\ar[rrr]& &&A^{op} 
}
$$

$$
T(\cross \cup s_b) =   \Ev\circ  ( \Phi_b^r )_1\circ (\coev_{1,2} \circ \ev_{1,2}) \circ 
(\Phi_a^r \circ\Phi_b \circ\ \Phi_a)_1 \circ \Coev
$$
$$
\xymatrix{
\ar@/_2pc/@{-}[d]A \ar[r]^-{\Phi_a } & A \ar[r]^-{\Phi_b} & A \ar[r]^-{\Phi_a^r } &
A \ar@/^2pc/@{-}[d] && \ar@/_2pc/@{-}[d]    A \ar[r]^-{\Phi_b^r }  & A \ar@/^3.5pc/@{-}[ddd]\\
  A^{op}\ar[rrr] &&& A^{op}   && A^{op}\ar[r]& A^{op}  \ar@/^1.5pc/@{-}[d]\\
\ar@/_2pc/@{-}[d]A^{op} \ar[rrrrrr]&&&&&&A^{op}\\
 A \ar[rrrrrr] &&&&&& A
}
$$
\end{defn}

All of the morphisms involving the above objects are straightforward concatenations of our previously defined morphisms
and all relations evidently hold.

Finally, we can further define the functor $T$ on the objects $s_a \cup s_b$ and $s_a\cup\cross\cup s_b$.

\begin{defn}
We define  
$$ 
\xymatrix{
T(s_a\cup s_b) =\Ev\circ  (\Phi_b^r)_1 \circ (R_{1,2} \circ \ev_{1,2})\circ
( \Phi_b
\circ \Phi_a^r )_1 \circ \ell_3    
\circ (\coev_{1,4}\circ L'_{1,4}) \circ (\Phi_a)_1\circ \Coev
}$$
$$ 
\xymatrix{
T(s_a\cup\cross\cup s_b) =\Ev\circ  ( \Phi_b^r )_1 \circ (R_{1,2} \circ \ev_{1,2})\circ
( \Phi_a^r
\circ \Phi_b)_1\circ  \ell_3   
\circ (\coev_{1,4}\circ L'_{1,4})  \circ (\Phi_a)_1\circ \Coev
}$$
\end{defn}

\begin{remark}
Once again, using nothing but  the basic properties of $\ev_A, \coev_A$, and canceling
the Serre equivalence with its inverse,
we have  the following simplifications
$$
\xymatrix{
T(s_a\cup s_b) \simeq\Ev\circ  ( \Phi_b^r)_1 \circ (\coev_{1,2} \circ \ev_{1,2})\circ
( \Phi_b
\circ \Phi_a^r)_1 \circ r_3 
\circ (\coev_{1,4} \circ \ev_{1,4}) \circ (\Phi_a)_1 \circ \Coev
}
$$
$$
\xymatrix{
T(s_a\cup\cross\cup s_b) \simeq\Ev\circ  ( \Phi_b)_1 \circ (\coev_{1,2} \circ \ev_{1,2}) \circ
(  \Phi_a^r
\circ\Phi_b^r)_1 \circ r_3 
  \circ (\coev_{1,4} \circ \ev_{1,4}) \circ (\Phi_a)_1\circ \Coev
}
$$

We can depict $ T(s_a\cup s_b)$  as the four-strand picture
$$
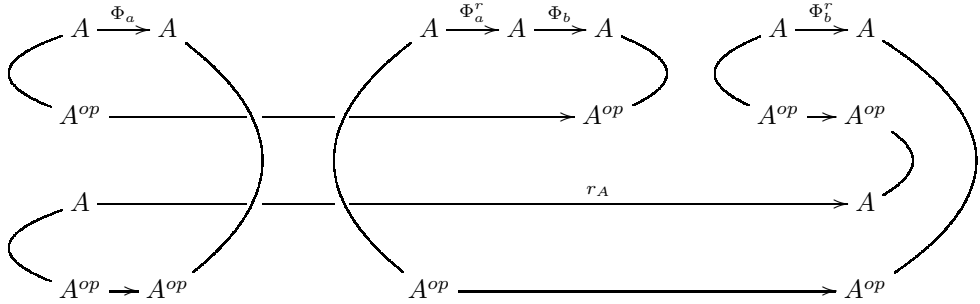
\xymatrix@!=0.75pc{
\ar@/_2.25pc/@{-}[d]A \ar[r]^-{\Phi_a } &A \ar@/^3pc/@{-}[ddd] 
&&& 
\ar@/_3pc/@{-}[ddd]  A \ar[r]^-{\Phi_a^r }  &A \ar[r]^-{\Phi_b }  & A \ar@/^2pc/@{-}[d] && \ar@/_2pc/@{-}[d]   
A  \ar[r]^-{\Phi_b^r }  &   A \ar@/^3.5pc/@{-}[ddd] \\ 
A^{op} \ar'[rr]'[rrr][rrrrrr] &&&&&& A^{op} && A^{op} \ar[r] & A^{op} \ar@/^1.5pc/@{-}[d]\\
\ar@/_2.25pc/@{-}[d] A \ar'[rr]'[rrr][rrrrrrrrr]^-{r_A} &&&&&&&&& A\\
  A^{op}\ar[r] & A^{op} &  && A^{op}\ar[rrrrr]& && && A^{op} 
}
$$
Similarly, we can depict  $ T(s_a\cup \cross\cup s_b)$  as the four-strand picture
$$
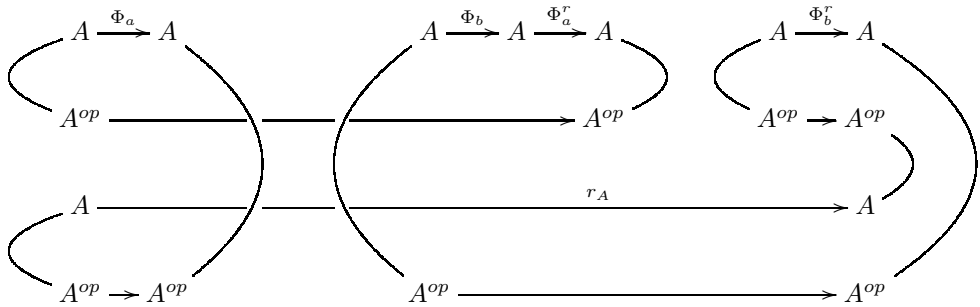
\xymatrix@!=0.75pc{
\ar@/_2.25pc/@{-}[d]A \ar[r]^-{\Phi_a } &A \ar@/^3pc/@{-}[ddd] 
&&& 
\ar@/_3pc/@{-}[ddd]  A \ar[r]^-{\Phi_b }  &A \ar[r]^-{\Phi_a^r }  & A \ar@/^2pc/@{-}[d] && \ar@/_2pc/@{-}[d]   
A  \ar[r]^-{\Phi_b^r }  &   A \ar@/^3.5pc/@{-}[ddd]  \\ 
A^{op} \ar'[rr]'[rrr][rrrrrr] &&&&&& A^{op} && A^{op} \ar[r] & A^{op}  \ar@/^1.5pc/@{-}[d] \\
\ar@/_2.25pc/@{-}[d] A \ar'[rr]'[rrr][rrrrrrrrr]^-{r_A}&&&&&&&&& A\\
  A^{op}\ar[r] & A^{op} &  && A^{op}\ar[rrrrr]& && && A^{op} 
}
$$
\end{remark}

Once again,
all of the morphisms involving the above objects are straightforward concatenations of our previously defined morphisms
and all relations evidently hold.

%

\subsubsection{Stratified maxima}
Set $\M_a, \M_b\in \cQ$ to be the saturated subsets
$$
\M_a = \cP \setminus\{(i, 1), (1,1)\}, \M_b = \cP \setminus\{(1, i), (1,1)\} \subset \cP.
$$

\begin{defn}
For $k=a,b$, we set $T(\M_k) = T(s_k)$.
\end{defn}

Observe that $s_a\cup \cross$ is terminal in $\cQ$ among objects less than $\M_a$,
and similarly $ \cross\cup s_b$ is terminal among objects less than $\M_b$.
Thus to define the functor $T$ on maps to $\M_a$ and $\M_b$, it suffices
to describe the morphisms $T(s_a\cup \cross, \M_b)$ and $T( \cross\cup s_b, \M_b)$. 

To this end, let us construct canonical equivalences 
$$
\xymatrix{
\mu_a:T(s_a \cup \cross) \ar[r]^-\sim& 
\Ev\circ  (  \Phi_a^r\circ \Phi_b\circ \Phi_b^r)_1\circ (\coev_{1,4} \circ \ev_{1,4}) \circ (\Phi_a)_1\circ \Coev
}
$$
$$
\xymatrix{
\mu_b:T( \cross\cup s_b) \ar[r]^-{\sim} & 
\Ev\circ  ( \Phi_b^r)_1\circ (\coev_{1,2} \circ \ev_{1,2}) \circ 
(\Phi_b \circ\ \Phi_a\circ \Phi_a^r )_1\circ \Coev
}
$$
Compare the pictures of $T(s_a\cup\cross)$
and $T(\cross\cup  s_b)$ appearing in Definition~\ref{defofsomeunions} to
the following pictures of the respective asserted targets
$$
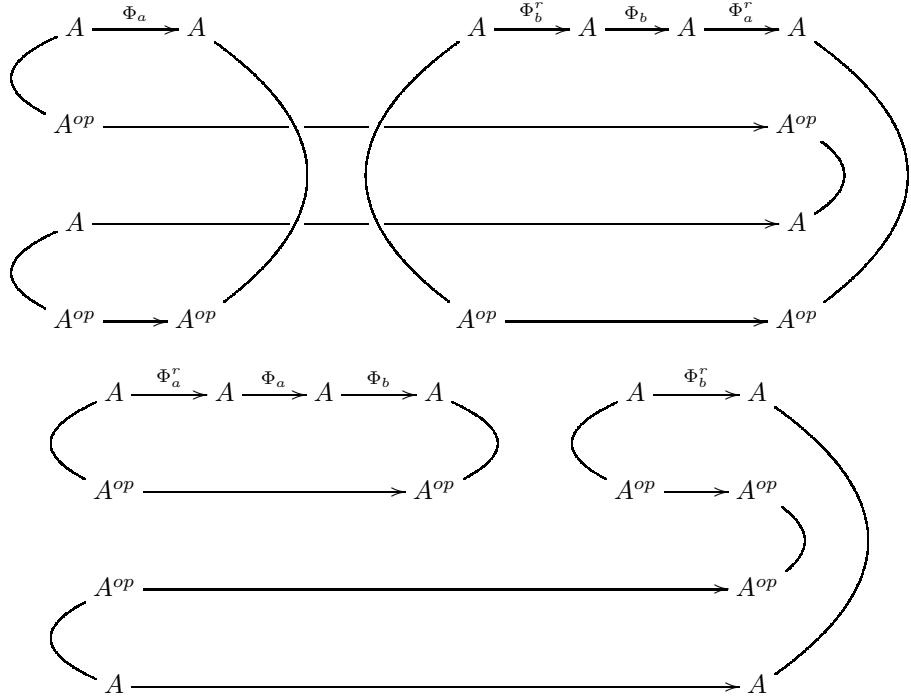
\xymatrix{
\ar@/_2pc/@{-}[d]A \ar[r]^-{\Phi_a } &A \ar@/^3.5pc/@{-}[ddd] 
&&& 
\ar@/_3.5pc/@{-}[ddd]  A \ar[r]^-{\Phi_b^r }  &A \ar[r]^-{\Phi_b }  &A \ar[r]^-{\Phi_a^r }  &A 
 \ar@/^3.5pc/@{-}[ddd] \\
A^{op} \ar'[rr]'[rrr][rrrrrrr] &&&&&&&A^{op}  \ar@/^1.5pc/@{-}[d]\\
 \ar@/_2pc/@{-}[d] A \ar'[rr]'[rrr][rrrrrrr] &&&&&&& A\\
  A^{op}\ar[r] & A^{op} &  && A^{op}\ar[rrr]& &&A^{op} 
}
$$
$$
\xymatrix{
\ar@/_2pc/@{-}[d]A \ar[r]^-{\Phi_a^r } & A \ar[r]^-{\Phi_a} & A \ar[r]^-{\Phi_b } &
A \ar@/^2pc/@{-}[d] && \ar@/_2pc/@{-}[d]    A \ar[r]^-{\Phi_b^r }  & A \ar@/^3.5pc/@{-}[ddd]\\
  A^{op}\ar[rrr] &&& A^{op} &  & A^{op}\ar[r]& A^{op}  \ar@/^1.5pc/@{-}[d]\\
\ar@/_2pc/@{-}[d]A^{op} \ar[rrrrrr]&&&&&&A^{op}\\
 A \ar[rrrrrr] &&&&&& A
}
$$
The sought-after equivalences $\mu_a$ and $\mu_b$ are obtained by taking the 
respective endomorphisms
$\Phi_b^r$ and $\Phi_a^r$  clockwise around the figures.

\begin{defn}
For $k=a,b$, we define the morphism
$$
\xymatrix{
T(s_k\cup \cross, \M_k) :T(s_k\cup \cross) \ar[r] & T(\M_k)
}
$$
to be induced by the composition of the equivalence $\mu_k$ constructed above
with the counit $\eps_{\Phi_{k'}}$ where $k'\in \{a,b\}$ is not equal to $k$.
\end{defn}

Now we can similarly define the functor $T$ on the other objects and morphisms of $\cQ$ that are unions
of those we have already encountered. 
In the following definition, we first state a formal expression for the value of the functor $T$ on such an object, and then depict 
the formal expression with a 
four-strand picture.

\begin{defn}
$$
T(\M_a \cup s_b)
=\Ev \circ (\coev_{1,2} \circ \ev_{1,2})
\circ (\Phi_a^r)_1 \circ r_3   
  \circ (\coev_{1,4}\circ \ev_{1,4}) \circ (\Phi_a)_1\circ \Coev
$$
$$
\xymatrix@!=0.75pc{
\ar@/_2.25pc/@{-}[d]A \ar[r]^-{\Phi_a } &A \ar@/^3pc/@{-}[ddd] 
&&& 
\ar@/_3pc/@{-}[ddd]    A \ar[r]^-{\Phi_a^r }  & A \ar@/^2pc/@{-}[d] && \ar@/_2pc/@{-}[d]   
  A \ar@/^3.5pc/@{-}[ddd]  \\ 
A^{op} \ar'[rr]'[rrr][rrrrr] &&&&& A^{op} && A^{op}   \ar@/^1.5pc/@{-}[d] \\
\ar@/_2.25pc/@{-}[d] A \ar'[rr]'[rrr][rrrrrrr]^-{r_A}&&&&&&& A\\
  A^{op}\ar[r] & A^{op} &  && A^{op}\ar[rrr]& &&  A^{op} 
}
$$

$$
T( s_a \cup \M_b)
=\Ev \circ (\Phi_b^r)_1 \circ (\coev_{1,2} \circ \ev_{1,2})
 \circ r_3   
  \circ (\Phi_b)_1 \circ(\coev_{1,4}\circ \ev_{1,4}) \circ \Coev
$$

$$
\xymatrix@!=0.75pc{
\ar@/_2.25pc/@{-}[d]A  \ar@/^3pc/@{-}[ddd] 
&&& 
\ar@/_3pc/@{-}[ddd]  A \ar[r]^-{\Phi_b }  &A  \ar@/^2pc/@{-}[d] && \ar@/_2pc/@{-}[d]   
A  \ar[r]^-{\Phi_b^r }  &   A \ar@/^3.5pc/@{-}[ddd]  \\ 
A^{op} \ar'[r]'[rr][rrrr] &&&& A^{op} && A^{op} \ar[r] & A^{op}  \ar@/^1.5pc/@{-}[d] \\
\ar@/_2.25pc/@{-}[d] A \ar'[r]'[rr][rrrrrrr]^-{r_A}&&&&&&& A\\
  A^{op}&  && A^{op}\ar[rrrr]& &&& A^{op} 
}
$$

$$
T(\M_a\cup \M_b) 
=\Ev \circ (\coev_{1,2} \circ \ev_{1,2})
 \circ r_3   
 \circ(\coev_{1,4}\circ \ev_{1,4}) \circ \Coev
$$
$$
\xymatrix@!=0.75pc{
\ar@/_2.25pc/@{-}[d]A  \ar@/^3pc/@{-}[ddd] 
&&& 
\ar@/_3pc/@{-}[ddd]  A   \ar@/^2pc/@{-}[d] && \ar@/_2pc/@{-}[d]   
   A \ar@/^3.5pc/@{-}[ddd]  \\ 
A^{op} \ar'[r]'[rr][rrr] &&& A^{op} && A^{op}   \ar@/^1.5pc/@{-}[d] \\
\ar@/_2.25pc/@{-}[d] A \ar'[r]'[rr][rrrrr]^-{r_A}&&&&& A\\
  A^{op}&  && A^{op}\ar[rr]& & A^{op} 
}
$$
\end{defn}

All of the morphisms involving the above objects are straightforward concatenations of our previously defined morphisms
and all relations evidently hold.


\subsubsection{Global maximum}
Finally, set $\M\in \cQ$ to be the entire poset $\M = \cP$. Recall that we set $T(\M) = 1_\cC$.

Since $\M_a\cup \M_b$ is terminal among  objects less than $\M$,
to complete the definition of the functor $T$, it suffices to define $T(\M_a\cup \M_b, \M)$.

To this end, observe that 
there is a canonical equivalence
$$
T(\M_a \cup \M_b)\simeq\ev_A \circ (\id_A \otimes r^{op}_A) \circ \coev_A
$$
In other words, we can picture $T(\M_1 \cup \M_2)$ in the form
$$
\xymatrix{
 \ar@/_2pc/@{-}[d] A \ar[r]^-{\id_{A}} &A \ar@/^2pc/@{-}[d] \\
A^{op} \ar[r]^-{r^{op}_A}&A^{op}\\
}
$$
where as usual the arcs denote the evaluation and coevaluation morphisms. 

In turn, 
 there are canonical equivalences
$$
\ev_A \circ R_A \simeq   \ev_A \circ (\id_A \otimes r^{op}_A) \circ \coev_A
\simeq L'_A\circ \coev_A
$$
which are compatible with the counit maps of the two adjoint pairs appearing.


\begin{defn} We define the morphism
$$
\xymatrix{
T(\M_a\cup \M_b, \M) = \eps:   \ev_A \circ R_A \ar[r] & 1_\cC
}$$
to be the counit of the adjunction.
\end{defn}

This concludes the construction of the functor $T:\cQ\to \cC$.


\subsection{Proof of main theorem}
 
 We seek to establish an equivalence
$$
\Tr_{\Tr(\Phi_a)}(\varphi(\Phi_a,\Phi_b, \alpha_{a,b}))  \simeq 
\Tr_{\Tr(\Phi_b^\vee)}(\varphi(\Phi_b^\vee,\Phi_a, \alpha_{b^\vee, a}))\in \End_\cC(1_\cC)
$$

We will show that each side can be interpreted as the morphism 
$$T(\emptyset, \M)\in  \End_\cC(1_\cC).
$$
More precisely, we will identify the left and right hand sides respectively with the compositions
$$
\xymatrix{
T(\es, \M) \simeq T(\M_k, \M) \circ T(s_k, \M_k) \circ T(\es, s_k),
& k =a,b,
}
$$
associated to the sequences
$$
\xymatrix{
\es\ar[r] &  s_k\ar[r]  & \M_k \ar[r] & \M,
& k =a,b.
}
$$

\begin{prop}\label{tracedual}
For $k=a,b$,  there is a natural equivalence
 $$
 \xymatrix{
  T(s_k)\simeq \Tr(\Phi_k) \otimes \Tr(\Phi_k)^\vee
 }$$ 
 and the respective morphisms
$$
\xymatrix{
1_\cC \ar[rr]^-{T(\es, s_k)} &&  T(s_k)\simeq \Tr(\Phi_k) \otimes \Tr(\Phi_k^\vee)
\simeq \Tr(\Phi_k) \otimes \Tr(\Phi_k)^\vee
}
$$
$$
\xymatrix{
 \Tr(\Phi_k) \otimes \Tr(\Phi_k)^\vee \simeq \Tr(\Phi_k) \otimes \Tr(\Phi_k^\vee)
\simeq 
T(s_k) = T(\M_k) \ar[rr]^-{T(\M_k , \M)} &&  1_\cC
}
$$
are the  natural coevaluation and evaluation morphisms.
\end{prop}

\begin{proof}

To identify the morphisms, recall the two equivalent proofs of Proposition~\ref{cylinders}.
For a right dualizable morphism $\Phi:A\to A$, the first proof presents 
 the evaluation of the adjoint pair $(\Tr(\Phi) , \Tr(\Phi)^\vee)$ as
the counit of the composition of adjoint pairs
\begin{eqnarray}\label{counit}
\xymatrix{
(\ev_A \circ \Phi\circ \coev_A, R'_A \circ \Phi^r\circ R_A)
}
\end{eqnarray}

The second proof presents
the coevaluation of the adjoint pair $(\Tr(\Phi) , \Tr(\Phi^\vee)$ as
the unit of the composition of adjoint pairs
\begin{eqnarray}\label{unit}
\xymatrix{
(L'_A \circ \Phi\circ L_A, \ev_A \circ \Phi^r\circ \coev_A)
}
\end{eqnarray}

By construction, the morphism
$$
\xymatrix{
T(\es, s_k) \simeq T(\m_k, s_k) \circ T(\m, \m_k)\circ T(\es, \m)
}
$$
is canonically equivalent to the composition of the units of the terms of the adjoint pair~(\ref{unit})
applied to $\Phi_k$.
In particular, the object $T(s_k)$ is indeed $\Tr(\Phi_k) \otimes \Tr(\Phi_k)^\vee$.

Similarly, by construction, the morphism 
$$
\xymatrix{
T(\M_k, \M) \simeq T(\M_k \cup s_{k'}, \M) \circ T(\M_k \cup \m_{k'}, \M_k \cup s_{k'})\circ T(\M_k, \M_k \cup \m_{k'}),
}$$
where $k'\in\{a, b\}$ is not equal to $k$,
is canonically equivalent to the composition of the counits of the terms of the adjoint pair~(\ref{counit})
applied to $\Phi_k$.
\end{proof}

Thanks to the above proposition, for $k=a,b$,
we can regard
 the morphism
$$
\xymatrix{
T(s_k, \M_k): T(s_k)  \ar[r] & T(\M_k)
}
$$
as an endomorphism of $ \Tr(\Phi_k) \otimes \Tr(\Phi_k)^\vee\simeq   \Tr(\Phi_k) \otimes \Tr(\Phi_k^r)$.

\begin{prop}
There are canonical equivalences of endomorphisms
$$
\xymatrix{
T(s_a, \M_a) \simeq \varphi_{\Tr(\Phi_a^r)}(\Phi_b, \alpha_{b, a^r}) \otimes \id_{\Tr(\Phi_a)}
&
T(s_b, \M_b) \simeq  \id_{\Tr(\Phi_b^r)}\otimes \varphi_{\Tr(\Phi_b)} (\Phi_a, \alpha_{a, b})
}$$
\end{prop}

\begin{proof}
For $k=a,b$, consider the terms of the composition
$$
\xymatrix{
T(s_k, \M_k) \simeq T(s_k \cup \cross, \M_k) \circ T(s_k \cup \m_{k'} ) \circ T(s_k, s_k \cup \m_{k'})
}$$
where $k'\in\{a, b\}$ is not equal to $k$.

For $k=a$, if we set $\Phi = \Phi_a^r$, $\Psi = \Phi_b$, and $\alpha =\alpha_{b, a^r}$, the
three 
terms are precisely 
  the three terms in Definition~\ref{def:trace funct} of the morphism
 $\varphi(\Psi, \alpha)$,
tensored with the identity of the factor $\Tr(\Phi_a)$.

For $k=b$, 
if we set $\Phi = \Phi_b$, $\Psi = \Phi_a$, and $\alpha= \alpha_{a, b}$, the
three terms are precisely the three terms in the alternative presentation of Remark~\ref{rem:alt trace funct} of the morphism
 $\varphi(\Psi, \alpha)$,
tensored with the identity of the factor $\Tr(\Phi_b^r)$.
\end{proof}

Taking the above two propositions together concludes the proof of the theorem.




\begin{thebibliography}{99}
   
\bibitem[BFN]{BFN} D. Ben-Zvi, J. Francis and D. Nadler, Integral transforms
and Drinfeld centers in derived algebraic geometry. Preprint arXiv:0805.0157.
{\em Jour. Amer. Math. Soc.} {\bf 23} (2010), 909--966.


\bibitem[BN09]{character} D. Ben-Zvi and D. Nadler, The character theory
  of a complex group.  e-print arXiv:math/0904.1247

\bibitem[BN10]{conns} D. Ben-Zvi and D. Nadler, Loop Spaces and
  Connections.  arXiv:1002.3636. To appear, Journal of Topology.


\bibitem[BN13]{primary} D. Ben-Zvi and D. Nadler, Nonlinear Traces.
e-print arXiv:1305.7175.


\bibitem[Cal1]{Cal1} A.~Caldararu, The Mukai pairing, I: The Hochschild structure,
arXiv:math/0308079v2.

\bibitem[Cal2]{Cal2} A.~Caldararu, The Mukai pairing, II:
The Hochschild-Kostant-Rosenberg
isomorphism,  Adv. Math.  194  (2005),  no. 1, 34--66, arXiv:math/0308080v3.


\bibitem[CT]{cisinskitabuada} D.-C. Cisinski and G. Tabuada,
 Lefschetz and Hirzebruch-Riemann-Roch formulas via noncommutative motives.
   arXiv:1111.0257.
   
   
\bibitem[DP]{doldpuppe}  A. Dold and D. Puppe, Duality, trace, and transfer. 
Proceedings of the International Conference on Geometric Topology (Warsaw, 1978), pp. 81ะ102, PWN, Warsaw, 1980.
   


\bibitem[DG]{finiteness} V. Drinfeld and D. Gaitsgory, On some finiteness questions for algebraic stacks.
Preprint arXiv:1108.5351.

\bibitem[EKMM]{EKMM} A. Elmendorf, I. Kriz, M. Mandell and J.P. May,
Rings, modules, and algebras in stable homotopy theory. With an appendix by M. Cole.
Mathematical Surveys and Monographs, 47.
American Mathematical Society, Providence, RI, 1997.


\bibitem[FG]{koszul} J. Francis and D. Gaitsgory, Chiral Koszul
  Duality. To appear, Selecta Math. arXiv:math/??
  
\bibitem[FN]{FN} E. Frenkel and Ng\^o B.-C.,  Geometrization of trace formulas. Bull. Math. Sci. 1 (2011), no. 1, 129ะ199.

\bibitem[G1]{indcoh} D. Gaitsgory, Ind-coherent
  sheaves. arXiv:math/1105.4857



\bibitem[G2]{1affine} D. Gaitsgory, Sheaves of Categories over
  Prestacks. Preprint, 2012.

\bibitem[GR1]{crystals} D. Gaitsgory and N. Rozenblyum, Notes on Geometric Langlands:
Crystals and $\D$-modules. Available at http://math.harvard.edu/\~{}gaitsgde/GL/crystaltext.pdf

\bibitem[GR2]{nickbook} D. Gaitsgory and N. Rozenblyum, Book in preparation.

\bibitem[GK]{GK} N. Ganter and M. Kapranov, Representation and character theory in 2-categories.
e-pring arXiv:math/0602510. Adv. Math. 217 (2008), no. 5, 2268-2300.
    
    
\bibitem[HKR]{HKR} M. Hopkins, N. Kuhn, and D. Ravenel, Generalized 
group characters and complex oriented cohomology theories. J.
Amer. Math. Soc., 13(3):553ะ594 (electronic), 2000.
    
    
\bibitem[Lo]{Loday} J.-L. Loday, Cyclic homology.
Appendix E by Mar\'ia O. Ronco.
Second edition.
Chapter 13 by the author in collaboration with Teimuraz Pirashvili.
Grundlehren der Mathematischen Wissenschaften [Fundamental Principles
 of Mathematical Sciences], 301. Springer-Verlag, Berlin,  1998.


\bibitem[Lu]{lunts} V. Lunts, Lefschetz fixed point theorems for
  Fourier-Mukai functors and DG algebras. arXiv:1102.2884
   
   
\bibitem[L1]{topos} J. Lurie, Higher topos theory.
arXiv:math.CT/0608040.

\bibitem[L2]{HA} J. Lurie, Higher Algebra. Available at \begin{verbatim} http://www.math.harvard.edu/\~lurie/  \end{verbatim}

\bibitem[L3]{TFT} J. Lurie, On the classification of topological
  field theories.  Available at http://www.math.harvard.edu/\~{}lurie/
  Current developments in mathematics, 2008, 129รยกV280, Int. Press,
  Somerville, MA, 2009.

\bibitem[L4]{2-cats} J. Lurie, $(\infty,2)$-Categories and the Goodwillie Calculus I. 
Preprint, available at http://www-math.mit.edu/\~{}lurie/

\bibitem[L5]{dag11} J. Lurie, Derived Algebraic Geometry XI: Descent
  theorems. Available at http://www.math.harvard.edu/\~{}lurie/



\bibitem[Ma]{markarian} N. Markarian, The Atiyah class, Hochschild
  cohomology and the Riemann-Roch theorem. arXiv:math/0610553 

\bibitem[M]{may}  J. P. May, Picard groups, Grothendieck rings, and Burnside rings of categories. Adv. Math. 163 (2001), no. 1, 1ะ16. 


\bibitem[Pe]{Petit} F.~Petit, A Riemann-Roch theorem for DG algebras, arXiv:1004.0361.

\bibitem[Po]{polishchuk} A. Polishchuk, Lefschetz type formulas for
  dg-categories arXiv:1111.0728.

\bibitem[PS1]{PS1} K. Ponto and M. Shulman,
Shadows and traces in bicategories, arXiv:0910.1306.

\bibitem[PS2]{PS2} K. Ponto and M. Shulman,
    Traces in symmetric monoidal categories. arXiv:1107.6032



\bibitem[P]{toly} A. Preygel, Thom-Sebastiani and duality for
  matrix factorizations. arXiv:math/1101.5834



\bibitem[Ram]{Ram} A.~C.~Ramadoss, The relative Riemann-Roch theorem from
Hochschild homology, New York J. Math. 14 (2008) 643-717.

\bibitem[Ram2]{Ram2} A.~C.~Ramadoss, The Mukai pairing and integral transforms in
Hochschild homology, Moscow Math. Journal, Vol 10, No 3, (2010), 629-645.

\bibitem[Shk]{Shk} D.~Shklyarov, Hirzebruch-Riemann-Roch theorem for
DG algebras, arXiv:0710.1937v3.


\bibitem[Sh]{shipley THH} B. Shipley, Symmetric spectra and topological Hochschild homology.  $K$-Theory  19  (2000),  no. 2, 155--183.

\bibitem[Shk]{Shk} D.~Shklyarov, Hirzebruch-Riemann-Roch theorem for
DG algebras, arXiv:0710.1937v3.


\bibitem[To]{Toen dg} B. To\"en, The homotopy theory of dg
categories and derived Morita theory. arXiv:math.AG/0408337. Invent.
Math.  167  (2007),  no. 3, 615--667.


\bibitem[To2]{Toen} B. To\"en, Higher and Derived Stacks: a global overview.
To appear, Proceedings 2005 AMS Summer School in Algebraic Geometry.
arXiv:math.AG/0604504.


\bibitem[TV]{TV} B. To\"en and G. Vezzosi, Infinies-categories monoidales rigides, traces et caracteres de Chern
    e-print arXiv:0903.3292












\end{thebibliography}
\end{document}